\documentclass[11pt,a4paper]{article}
\usepackage{amsmath,amssymb,latexsym}
\usepackage{amsthm}
\usepackage{bm}
\usepackage{overpic}
\usepackage{color}
\theoremstyle{plain}% default
\newtheorem{theorem}{Theorem}[section]
\newtheorem{corollary}{Corollary}[section]

\newtheorem{lemma}{Lemma}[section]
\theoremstyle{definition}

\theoremstyle{remark}
\newtheorem{remark}{Remark}[section]
\newcommand{\res}{\mathop{\rm res}}
\newcommand{\cp}{\mathop{\rm cap}}
\newcommand{\supp}{\mathop{\rm supp}}

\newcommand{\field}[1]{\mathbb{#1}}

\newcommand{\R}{\field{R}}

\newcommand{\N}{\field{N}}
\newcommand{\C}{\field{C}}

\newcommand{\PP}{{\mathcal P}}

\renewcommand{\Re}{\mathop{\rm Re}}
\renewcommand{\Im}{\mathop{\rm Im}}

\newcommand{\dsty}{\displaystyle}

\def\XXint#1#2#3{{\setbox0=\hbox{$#1{#2#3}{\int}$}
\vcenter{\hbox{$#2#3$}}\kern-.5\wd0}}

\title{% On the dynamics of the equilibrium measure for a polynomial external field
Equilibrium measures in the presence of certain rational external fields}

\author{R. Orive and J. S\'{a}nchez Lara}

\date{\today}

\begin{document}

\vspace{1cm} \maketitle

\begin{abstract}
Equilibrium measures in the real axis in the presence of rational external fields are considered. These external fields are called rational since their derivatives are rational functions. We analyze the evolution of the equilibrium measure, and its support, when the size of the measure, $t$, or other parameters in the external field vary. Our analysis is illustrated by studying with detail the case of a generalized Gauss-Penner model, which, in addition to its mathematical relevance, has important physical applications (in the framework of random matrix models). This paper is a natural continuation of \cite{MOR2013}, where equilibrium measures in the presence of polynomial external fields are thoroughly studied.
\end{abstract}
\vspace{0.5cm}

\section{Introduction}
This paper deals with the study of families of equilibrium measures on the real line in the rational external fields. In particular, these measures are considered as  functions of a parameter $t$ representing either the total mass, time (in the context of dynamical systems) or temperature (in a physical sense), following the method used in \cite{MOR2013} for the case of polynomial external fields. These families are regarded as models of many physical processes, which motivates their intense study in the context of the mathematical physics, mainly within the framework of random matrix models. But equilibrium problems in the presence of external fields play an important role also in approximation theory; in particular, in the asymptotics of orthogonal and Heine-Stieljes polynomials and, in connection with the former, in the study of the convergence of sequences of rational interpolants (Pad\'{e} and Pad\'{e}-type approximants).

Next, we introduce the class of external fields we are interested, as well as the main tools for the analysis carried out in Section 2 (see Theorems 1.1 - 1.3 below). To end this section, some applications in approximation theory and random matrices models will be briefly recalled.

\subsection{Equilibrium problems in the rational external fields}
In \cite[sect. 5]{Rakh2012}, E. A. Rakhmanov considers an important class of external fields, namely
\begin{equation}\label{rakhrat}
\varphi (z) = \sum_{j=1}^m\,\gamma_j\,\log\left|\frac{1}{z-z_j}\right|\,,\;\gamma_j\in \R\,,\;z_j\in \C\,.
\end{equation}
These are external fields generated by a system of fixed charges in the set $\emph{A} = \{z_1,\ldots,z_m\}\,.$ They are called rational external fields since the derivative of the complex field $\Phi$ whose real part is $\varphi$ is given by
\begin{equation*}\label{field derivative}
\Phi'(z) = -\,\sum_{j=1}^m\,\frac{\gamma_j}{z-z_j}\,=\,\frac{U(z)}{V(z)}\,,
\end{equation*}
which is a rational function with denominator $V(x) = \prod_{j=1}^m\,(z-z_j)$ of exact degree $m$ and numerator $U$, a polynomial of degree $\leq m-1$ such that $\res_{z=z_j}\, \frac{U(z)}{V(z)} = \gamma_j\,,\;j=1,\ldots,m\,.$ In \cite{OrGa2010}, a particular case of rational external fields \eqref{rakhrat} is handled, namely when all the charges $\gamma_j$ are negative (that is, ``attractive'') and placed in $\C\setminus\R$. There, this class of rational external fields is considered in connection with the asymptotics of sequences of generalized Heine-Stieljes polynomials; see \cite{MFSa2002} for the asymptotics of Heine-Stieltjes polynomials in the classical setting, where all the charges are positive and placed in the real axis, and \cite{MFRa2011} for a more general setting.

The aim of this paper is also related to a physical context: the study of the limit distribution of eigenvalues in random matrix models. In the latter scenario, some classes of rational external fields have been recently considered. For instance, in \cite{Kris2006} the external field $\varphi (x) = x^4 - \log(x^2 +v)\,,\,v > 0\,,$ whose derivative is clearly a rational function (now  the degree of the numerator is greater than the degree of the denominator), is considered in $\R$ in order to provide a toy-model for studying the gluon self-coupling and their interactions with quarks in a baryon. In \cite{Kris2006}, the author is also concerned with the situation when $v \searrow 0\,,$ which leads to the so-called generalized Penner model, for which $\varphi(x) = ax^4+bx^2-c\ln|x|\,,$ extending the classical Gauss-Pener model where $\varphi(x) = x^2-k\ln|x|\,$ (see e.g. \cite{Deo2002} for details).

In this paper, one of our main aims is to study the evolution of the equilibrium measure in the real axis in the presence of rational external fields of the type
\begin{equation}\label{quarticperturblog}
\varphi(x) = P(x)+\sum_{j=1}^{q}\,\alpha_j\,\log |x-z_j|\,,\;q\geq 1\,,\,z_j\in \C \setminus \R\,,\,\alpha_j \in \R\,,
\end{equation}
where $P$ is a polynomial of degree $2p$ (the case $P\equiv 0\,$ is allowed) and it is assumed that at least one of the $\alpha_j$ is different from zero. With this evolution we mean the variation of the equilibrium measure $\lambda$ and its support when its size, $\lambda (\R) = t>0$, or other parameters appearing in the external field vary (below, some definitions and basic properties concerning equilibrium measures in the external fields will be recalled). In this sense, when $p\geq 1$, or what is the same, when the degree of the numerator of the rational function $\varphi'$ is bigger than the degree of the denominator, then we are in the presence of a polynomial external field perturbed by the action of some (``attractive'' or ``repulsive'') fixed charges. In this case, the evolution of $\lambda = \lambda_t$ when $t$ grows from $0$ to $+\infty$ may be studied. The current paper is mainly devoted to this situation. On the other hand, when $P\equiv 0$, that is, when we are given a ``purely rational'' external field, in the sense that it is exclusively due to the action of some fixed charges, we need to impose the extra condition that $\dsty \sum_{j=1}^{q}\,\alpha_j\,>\,t$ in order to guarantee the admissibility of the external field (or, what is the same, the compactness of the support of the equilibrium measure, see e.g. \cite{Saff:97}). Thus, in this particular case just a bounded set of values of $t$ may be considered. Because of this circumstance, this case will be studied with more detail in a forthcoming paper. Note also that when $p=0$, since the equilibrium problem is invariant by addition of constants in the external field, it may be assumed that $P\equiv 0\,.$

In the current paper, along with general considerations about the equilibrium measures in the presence of general rational external fields, we study in detail the simple but illustrative case, important for its applications (see \cite{Kris2006}), where $P$ is an even polynomial of exact degree $4$ and we have a couple of fixed charges placed at conjugate points in the imaginary axis, as the total mass of the measure (also, the ``time'' or ``temperature'', in the dynamical systems or physical, respectively, frameworks) grows from $0$ to $\infty$. Thus, this external field is symmetric with respect to the imaginary axis and this symmetry allows us to study an equivalent simplified problem in the positive semi axis $[0,+\infty)$, which also has an independent interest; for instance, it allows to consider the case where the support has the so-called hard-edges at an endpoint. Through the analysis of this illustrative example, other of our main aims in this paper will be achieved, namely, the study of the dynamics of the equilibrium measure when other parameters of the problem, different from the total mass $t$, vary; in particular, the dynamics when the prescribed charge approaches the real axis is handled.
Recall that the ``quartic'' external field was studied with detail in \cite{MOR2013}, in the framework of general polynomial external fields. On the other hand, the polynomial part of the rational field analyzed in the generalized Gauss-Penner model considered in the present paper has been studied in several previous papers (see e.g. \cite{Bleher99} and \cite{Bleher/Its03}, among others).

%\subsection{Families of equilibrium measures in the real axis in the presence of real analytic external fields}

Next, we introduce basic notations and recall a number of well-known properties of the equilibrium measures on the real line. Since this paper is a continuation of \cite{MOR2013}, the reader is referred to this reference for more details, as well as to the monograph \cite{Saff:97}.
%For more details the reader can consult the original papers  \cite{Buyarov/Rakhmanov:99, GR:87, MR638916, MR675192} and

Since the external field \eqref{quarticperturblog} satisfies the condition
\begin{equation} \label{cond2}
\varliminf\limits_{|x|\to
+\infty} \frac{\varphi(x)}{\log |x|}=+\infty\,,
\end{equation}	
given any $t>0$, there exists a measure $\lambda_t = \lambda_{t,\varphi}$, such that $\lambda_t (\R) = t$, with compact support $S_t \subset \R\,,$ uniquely determined by the equilibrium condition (see e.g. \cite{Saff:97})
\begin{equation}\label{equilibrium}
W_\varphi^{\lambda_t}(z)= V^{\lambda_t}(x) + \varphi (x)\,\begin{cases} = c_t\,,\;& x\in S_t\,, \\ \geq c_t\,,\;& x\in \R\,, \end{cases}
\end{equation}
where for a measure $\sigma\,,$ $\dsty V^{\sigma}(x) = - \int \,\log |x-s|\,d\sigma(s)\,.$ Measure $\lambda_t$ is called the equilibrium measure in the presence of $\varphi$ and minimizes the weighted energy $$I_{\varphi}(\sigma) = -\int \int \log |x-z|\,d\sigma(x)\,d\sigma(z) + 2\int\varphi (x) d\sigma (x)\,,$$
among all measures $\sigma$ supported in the real axis and such that $\sigma (\R) = t\,.$

When we consider the conductor $\Sigma = \R^+ = [0,+\infty)$ instead of the whole real line, the external field $\varphi$ can be extended as an external field in $\R$ by defining $\varphi (x) = +\infty\,,x\in (-\infty,0)\,$; this will be the situation in Section 2 below.
In a similar fashion, when  $\Sigma = K$ is a compact  set on $\R$ and  the external field $\varphi$ is
$$
\varphi(x) =\begin{cases} 0, & \text{if } x\in K, \\
+\infty , & \text{otherwise},
\end{cases}
$$
the corresponding energy minimizer $\lambda_{1}$  in the class of all probability (unit) measures, denoted by $\omega_K$, is known as the \emph{Robin measure} of $K$, its energy   is the \emph{Robin constant}   of $K$,
\begin{equation*} \label{robin}	
\rho(K) = I(\omega_K) = -\int \int \log |x-z|\,d\sigma(x)\,d\sigma(z),
\end{equation*}
and the logarithmic capacity of $K$ is given by $ \cp (K)=\exp(-\rho(K))$.
Under quite mild conditions on $K$ (e.g., regularity with respect to Dirichlet problem, see \cite{Saff:97}), measure $\omega_K$ is characterized by the the fact that  $\supp(\omega_K)=K$ together with the equilibrium condition
\begin{equation*} \label{equilibriumRobin}	
V^{\omega_K}(z) \equiv \rho (K), \quad z \in K.
\end{equation*}
However, for a general external field the support of the equilibrium measure is not known in advance and has to be found from the equilibrium conditions; this is the main problem to determine the equilibrium measure.

In the case of external fields $\varphi$ given by \eqref{quarticperturblog}, it may be assumed without loss of generality that $\dsty P(x) = \frac{x^{2p}}{2p}\,+\sum_{j=1}^{2p-1}\,c_j\,x^j\,,\;p\geq 1\,,$ and then the derivative of $\varphi$ is a rational function of the form:
\begin{equation*}\label{derivef}
\varphi'(x) = P'(x) + \sum_{j=1}^{q}\,\frac{\alpha_j\,(x-\Re z_j)}{(x-z_j)(x-\overline{z_j})}\,=\,\frac{D(x)P'(x)+C(x)}{D(x)}\,=\,\frac{E(x)}{D(x)}\,,
\end{equation*}
with $\dsty D(x) = \prod_{j=1}^{q}\,(x-z_j)(x-\overline{z_j})\,$ and $C\in \PP_{2q-1}\,,$ such that $\dsty \sum_{j=1}^{q}\,\frac{\alpha_j\,(x-\Re z_j)}{(x-z_j)(x-\overline{z_j})}\,=\,\frac{C(x)}{D(x)}\,.$

Now, we are concerned with two main ingredients for the study of the equilibrium measures in the external field \eqref{quarticperturblog}. The first one is an algebraic equation satisfied by the Cauchy transform of the equilibrium measure $\lambda_t$. In this sense, as an immediate consequence of \cite[Theorem 2]{MOR2013} we have
\begin{theorem}\label{Rrational}
For the rational external field \eqref{quarticperturblog}, the associated equilibrium measure, $\lambda_t$, and its Cauchy transform $\displaystyle\widehat{\lambda_t}(z)=\int \frac{d\lambda_t(s)}{z-s}$, it holds\vspace{-0.2cm}
\begin{equation}\label{Cauchytr}
(-\widehat{\lambda_t} + \varphi')\,(z)\,=\,\sqrt{R(z)}\,=\,\frac{B(z)\,\sqrt{A(z)}}{D(z)}\,,\;z\in \C \setminus S_t\,,
\end{equation}
for polynomials $\dsty A(z) = \prod_{j=1}^{2k}\,(z-a_j)$ and $B(z) = \dsty \prod_{j=1}^{2(p+q)-k-1}(z-b_j)$, with $a_1<\ldots <a_{2k} \in \R\,,$ $1\leq k\leq p+q$. We consider the branch of $\sqrt{A(z)}$ such that $\sqrt{A(z)} > 0$ for $z \in (a_{2k}, +\infty) \subset \R$.
\end{theorem}
Of course, points $a_j$ and $b_j$ and, consequently, functions $A,B$ and $R$ depend on $t$, but we will omit this dependence in our notation for simplicity. The zeros of $A$ (zeros of odd order of $R$) are the endpoints of the $k$ intervals comprising the support $S_t$ of $\lambda_t$; observe that the number of such intervals is given by $k$, with $1\leq k\leq p+q$.

Identity \eqref{Cauchytr} also provides the expression for the density of the equilibrium measure $\lambda_t$,
\begin{equation}\label{density}
 d   \lambda_t(x)=\frac{1}{\pi}\,\sqrt{|R(x)|}\,dx,\quad x\in S_t\,,
    \end{equation}
and its chemical potential,
\begin{equation}\label{totalpotential}
    W_\varphi^{\lambda_t}(z)= \Re \,\int^z\,\sqrt{R(x)}\,dx\,,
    \end{equation}
with an appropriate choice of the lower limit of integration and of the branch of the square root. In addition, by equilibrium condition \eqref{equilibrium}, \eqref{totalpotential} implies that
\begin{align}
\label{periods}
\int_{a_{2j}}^{a_{2j+1}}\,\sqrt{R(x)}\,dx = & 0\,, \quad  j=1,\cdots ,k-1\,,
\end{align}
which means that $B$ has at least a zero of odd order on each gap $(a_{2j},a_{2j+1})$ of $S_t\,.$ In the simplest case when $k=1$ (connected support: ``one-cut''), relation \eqref{Cauchytr} allows us to get a system of nonlinear equations whose solutions are the parameters of the problem, that is, the zeros of polynomials $A$ (endpoints of the support) and $B$ (remainder zeros of the density function of the equilibrium measure). When $k>1$, we need to make use of extra conditions \eqref{periods}, and thus, complicated hyperelliptic integrals are involved. This is the reason why it is not possible to obtain a totally explicit expression of the density of the equilibrium measure in the general multi-cut case.

Other important ingredient for our analysis is the dynamical system for the parameters of the problem which could be derived from the seminal Buyarov-Rakhmanov result (see the original paper \cite{Buyarov/Rakhmanov:99} and \cite[Theorem 3]{MOR2013}), which basically asserts that at a certain ``instant'' $t_0$, the derivative of the equilibrium measure $\lambda_t$ with respect to $t$ is given by the Robin measure for the support $S_{t_0}$. In the rational case, taking into account these results and the well-known expression for the Robin measure of a finite union of compact intervals, we have that except for a denumerable set of values of $t$,
\begin{equation}\label{BuyRakhrat}
\frac{\partial}{\partial t}\,\left(\frac{B(z)\,\sqrt{A(z)}}{D(z)}\right)\,=\,-\,\frac{F(z)}{\sqrt{A(z)}}\,,
\end{equation}
where $F$ is a monic polynomial of degree $k-1$ such that $\dsty \int_{a_{2j}}^{a_{2j+1}}\,\frac{F(x)}{\sqrt{A(x)}}\,,dx\,=\,0\,,\;j=1,\ldots,k-1\,,$ which means that $F$ has a zero on each gap $(a_{2j},a_{2j+1})$ of the support. As in \cite{MOR2013}, we will make use of the abbreviate physical notation for the derivative with respect to the ``time'' $t$: $\dsty \dot{f} = \frac{\partial f}{\partial t}\,.$ Thus, using \eqref{BuyRakhrat}, one immediately obtains:
\begin{equation}\label{Withamrat}
2A(x)\dot{B}(x) + B(x)\dot{A}(x)  = -2D(x)F(x)\,,
\end{equation}
which is an extension of \cite[formula (52)]{MOR2013} for the rational case. Now, evaluating expression \eqref{Withamrat} in the zeros of $A$ and $B$, it yields
\begin{theorem}\label{dynamicalsystem}
Except for a denumerable set of values of $t$, it holds
\begin{equation}\label{variationzeros}
\begin{split}
\dot{a_j} & = \,\frac{2D(a_j)F(a_j)}{\prod_{k\neq j}(a_j-a_k)\,B(a_j)}\,,\,j=1,\ldots,2k \,,\\
\dot{b_j} & = \,\frac{D(b_j)F(b_j)}{\prod_{k\neq j}(b_j-b_k)\,A(b_j)}\,,\,j=1,\ldots,2(p+q)-k-1\,,
\end{split}
\end{equation}
with $1\leq k\leq p+q\,.$
\end{theorem}

Indeed, \eqref{variationzeros} is a (autonomous) dynamical system for the positions of all the important points determining the equilibrium measure and its support. In \cite{MOR2013} it was shown how a suitably combined use of \eqref{Cauchytr} and \eqref{variationzeros} can explain the whole dynamics of the problem for the polynomial case, specially the more ``intriguing'' phenomena of the phase transitions.
 Recall that for real-analytic external fields (see e.g. \cite{MOR2013}), the equilibrium measure and its support depend analitically of $t$ except for a denumerable set of ``critical'' values of $t$, the singularities. These singularities take place when collisions/bifurcations of zeros of polynomials $A$ and/or $B$ occur. In the case of real analytic external fields, just three types of generic singularities can occur, namely (using the classification in \cite{MR2001g:42050} and \cite{KuML 00}),

\begin{itemize}
\item \textbf{Singularity of type I:} at a time $t=T$ a real zero $b$ of $B$ (a double zero of $R_t$) splits into two simple zeros $a_-<a_+$, and the interval $[a_-,a_+]$ becomes part of $S_t$ (\emph{birth of a cut}); we assume that for $t$ in a neighborhood of $t=T$, $b$ is a simple zero of $B$. This type of singularity occurs by saturation of inequality in \eqref{equilibrium};
 \item \textbf{Singularity of type II:} at a time $t=T$ two simple zeros $a_{2s}$ and $a_{2s+1}$ of $A$ (simple zeros of $R_t$) collide (\emph{fusion of two cuts} or \emph{closing of a gap}).
 \item \textbf{Singularity of type III:} at a time $t=T$ a pair of complex conjugate zeros $b$ and $\overline b$ of $B$ (double zeros of $R_t$) collide with a simple zero $a$ of $A$, so that $\lambda_T'(x)=\mathcal O(|x-a|^{5/2})$ as $x\to a$.
 \end{itemize}
In the case where the number of cuts is bounded by $2$, these are just all the possible types of singularities, while where it can be greater than $2$, more intriguing phenomena can occur when two or more of these singularities take place simultaneously. However, in \cite{MOR2013} it is pointed out that they are the basic ``building blocks'' of all phase transitions.

Though maybe the main aim of this paper is extending this analysis to the case of rational external fields (see Section 2.1 below), another aim, also very important for us, is the study of the evolution of the equilibrium measure when other parameters, different from the total mass $t$, vary.

In order to do it, in \cite{MOR2013} the authors extended the seminal Buyarov-Rakhmanov result for the variation of the equilibrium measure with respect to other parameters in the external field. Here, we make use of a simplified version of this result \cite[Theorem 5]{MOR2013}, which is sufficient for our purposes.
\begin{theorem}\label{parameters}
Let $t>0$ be fixed and suppose that the function $\varphi(x;\tau)$ is real-analytic for $x\in \R$ and $\tau \in (c,d)$, where $(c,d)$ is a real interval. Let $\lambda = \lambda_{t,\tau}$ denote the equilibrium measure in the external field $\varphi(x;\tau)$, for $\tau \in (c,d)$, with support $S_{t,\tau}$. Then, for any $\tau_0 \in (c,d)$, $$\frac{\partial \lambda}{\partial \tau}\,|_{\tau=\tau_0} = \omega\,,$$ where the measure $\omega$ is uniquely determined by the conditions
\begin{equation}\label{condparam}
\supp \omega = S_{t,\tau_0}\,,\;\omega (S_{t,\tau_0}) = 0\,,\;V^{\omega} + \frac{\partial \varphi(x;\tau)}{\partial \tau}\,|_{\tau=\tau_0} = \frac{\partial c_t}{\partial \tau}\,|_{\tau=\tau_0} = const\;\;\text{on}\;\; S_{t,\tau_0}
\end{equation}

\end{theorem}

Observe that the second formula in \eqref{condparam} means that $\omega$ is a type of signed measure which is often called a \emph{neutral} measure; it is a natural consequence of the fact that $t$, the total mass of $\lambda$, does not depend on parameter $\tau$. In \cite[Section 3.4]{MOR2013}, this result was used only for the one-cut case. One of our aims in the present paper is making use of this relations also when the support comprises two intervals (see Section 2.2 below).

In Section 2, an important particular case of equilibrium problems in the presence of rational external fields, the so-called generalized Gauss-Penner model will be treated with detail. First, some important applications will be recalled.

\subsection{Applications in approximation theory and random matrix models}
In this subsection, we restrict ourselves to include a brief description of some applications of the equilibrium problems in the presence of external fields. These applications are well known in the literature (see e.g. the extensive introduction of \cite{MOR2013}), but it seems convenient to provide this brief account to make the paper self-contained.

Our interest in studying equilibrium problems in the presence of external fields initially comes from their wide range of applications in Approximation Theory, in particular, those related to the asymptotic behavior of orthogonal and Heine-Stieltjes polynomials. However, the applications in Random Matrix Theory are also (and, even, more) important. Thus, we restrict ourselves to these two fields, in spite that there are also very interesting applications to Integrable Systems (in particular, the so-called Toda lattices, see \cite{DeML} and references therein) and in Soliton Theory \cite{LaxLev}.

Equilibrium problems in the external fields has become an essential tool for studying the asymptotic behavior of orthogonal polynomials with exponential or, in general, varying weights. Let us recall briefly some topics in this sense (for details, see \cite{Saff:97}).

Let $P_N = P_{N,n}$ be the monic polynomial of degree $N$ satisfying the orthogonality relations
\begin{equation*}\label{varyingorthog}
\int_{\R}\,x^j\,P_N(x)\,e^{-2n\varphi (x)}\,dx\,=\,0\,,\;j=0,\ldots,N-1\,,
\end{equation*}
where $\varphi$ is a continuous function satisfying \eqref{cond2}. It is often interesting studying the asymptotics of polynomials $P_{N,n}$ when both $N,n\rightarrow \infty$ in such a way that the ratio $\frac{N}{n}\rightarrow t>0\,.$ In particular, denoting by $\dsty \nu_n = \,\frac{1}{n} \sum_{P_N (\zeta)=0}\,\delta_{\zeta}$, it is well-known that it holds
$$\nu_n\,\stackrel{*}{\longrightarrow} \lambda_t(\varphi)\,,$$
where symbol $\stackrel{*}{\longrightarrow}$ stands, as usual, for weak-* convergence. This result extends, under rather mild conditions, to the case of more general varying weights $\omega_n = e^{-\phi_n(x)}$, provided that $\dsty \lim_{n\rightarrow \infty}\,\frac{1}{2n}\,\phi_n = \varphi\,.$ In this sense, external fields of the form \eqref{quarticperturblog} are related with a class of modified varying Freud weights, of the form
$$\omega(z)=e^{-2nP(z)}\prod_j |z-z_j|^{-\alpha_j n}\,.$$
If the exponents $\alpha_j$ are all positive integers, then we deal with a Geronimus transformation of the Freud weight, while when all $\alpha_j$ are negative integers, then a Christoffel transformation is considered. The results of this paper are related with
the dynamics for the limit distributions of the zeros of the varying orthogonal polynomials as some parameters of the generalized Freud weight above vary.

The discrete counterpart of the problem above is related to the so-called weighted Fekete points (see also \cite{Saff:97} for a deeper study of this topic). Given a continuous external field $\varphi$, for any configuration of points $(\zeta_1,\ldots,\zeta_n)\,,\;$ with $\zeta_1,\ldots,\zeta_n \in \R,$ and $\zeta_1 \leq \ldots \leq \zeta_n$, we consider the (discrete) weighted energy given by
\begin{equation}\label{discreteenergy}
\begin{split}
E_{\varphi}(\zeta_1,\ldots,\zeta_n) & = \sum_{i<j}\,\log \,\frac{1}{|\zeta_i - \zeta_j|}\,+\,\sum_{i=1}^n\,\varphi(\zeta_i)\\
& = \frac{1}{2}\,\left(\sum_{i\neq j}\,\log \,\frac{1}{|\zeta_i - \zeta_j|}\,+\,2\,\sum_{i=1}\,\varphi(\zeta_i)\right)\,.
\end{split}
\end{equation}
Then, if for each $n\in \N$, $(\zeta_1^*,\ldots,\zeta_n^*)$ is a minimal configuration for \eqref{discreteenergy}, we have (see \cite[Ch. III]{Saff:97}):
$$\nu_n = \frac{1}{n}\,\sum_{i=1}^n\,\delta_{\zeta_i^*}\,\stackrel{*}{\longrightarrow}\,\lambda_1(\varphi)\,.$$ Observe that, in general, this minimum is not unique; see, for instance, the example at the beginning of Section 2.2 in \cite{OrGa2010}. In fact, that example is related to one of the best known type of weighted Fekete points: the zeros of Heine-Stieltjes polynomials. Let us say a word about these polynomials. For certain rational external fields $\varphi$, the polynomials $\dsty y(x) = y_n(x) = \prod_{i=1}^n\,(x-\zeta_i^*)\,,$ whose zeros are the minimizers of \eqref{discreteenergy}, satisfy a second order linear differential equation with polynomial coefficients, $\dsty A_n y'' + B_n y' + C_n y = 0\,.$ Indeed, this is the case when $\dsty \varphi(x) = \sum_{j=1}^q\,\alpha_j\,\log\,|x-a_j|\,,$ with real $a_j\in \R$ and $\alpha_j <0\,,\;j=1,\ldots,q\,,$ which corresponds with a particular case of the rational fields considered in the current paper, namely, the ``purely'' rational external fields (that is, when $P\equiv 0$ in \eqref{quarticperturblog}). Polynomial solutions $y(x)=y_n (x)$ and polynomial coefficients $C_n(x)$ are called Heine-Stieltjes and Van Vleck polynomials, respectively and their study was built on the seminal works by T. J. Stieltjes (see e.g. \cite{Stieltjes1885}). The simplest situation, corresponding to the case when $q=2$, is related to the well-known Jacobi polynomials setting. Asymptotic properties of Heine-Stieltjes and Van Vleck polynomials were extensively studied in \cite{MFSa2002}, for the the classical Heine-Stieltjes setting.

On the other hand, in \cite{OrGa2010} the case where $\dsty a_j \in \C \setminus \R$ and $\alpha_j > 0\,,\;j=1,\ldots,q$ (that is, an ``attractive'' model) was studied, in the sense of both the discrete and continuous equilibrium problems. This problem is the natural real counterpart of an equilibrium problem in the unit circle previously analyzed in \cite{Gr 02} and \cite{MMO2005}.

Finally, let us note that for general external fields in \eqref{quarticperturblog}, i.e. with $P\neq 0$, or what is the same, when a polynomial external field adds to the effect of the mutual repulsion of the charges, the monic polynomial whose zeros are the corresponding weighted Fekete points also satisfy a second order linear differential equation with polynomial coefficients.

As we said above, the other large circle of applications of equilibrium measures in external fields of great interest from our viewpoint is that related to Random Matrix models. This is an important theory within the mathematical physics and, more precisely, the statistical mechanics. It is often assumed that this theory received a major boost from the efforts of E. Wigner in the 1950s for finding a mathematical model for the Hamiltonian of a heavy nucleus. Indeed, he found the behavior of the eigenvalues of Hermitian random matrices to be a suitable toy-model for this process (see e.g. \cite{Deiftbook} for some historical remarks in this sense).

Specifically, let us consider the set of $N\times N$ Hermitian matrices
$$
\left\{ M=\left( M_{jk}\right)_{j,k=1}^N:\; M_{kj}=\overline{M_{jk}} \right\}\,,
$$
as endowed with the  joint probability distribution
$$
d\nu_N(M)=\frac{1}{\widetilde Z_N}\, \exp\left( -  \mathrm{Tr} \, V(M)\right)\, dM, \qquad
dM=\prod_{j=1}^N dM_{jj}\prod_{j\not=k}^N d\Re M_{jk}d\Im M_{jk},
$$
where $V : \, \R \to \R$ is a given function such that
the integral in the definition of the normalizing constant
$$
\widetilde Z_N = \int \exp\left( -  \mathrm{Tr} \, V(M)\right)\, dM
$$
converges. Then, it is well-known (see e.g. \cite{Mehta2004}) that $\nu_N$ induces a joint probability distribution $\mu_N$ on the  eigenvalues $\lambda_1<\dots <\lambda_N$ of these matrices, with the density
\begin{equation*}
\label{Eigdensity}
\mu_N' (\bm \lambda)  = \frac{1}{ Z_N}\, \prod_{i<j} (\lambda_i-\lambda_j)^2 \exp\left( -  \sum_{i=1}^N V(\lambda_i)\right) ,
\end{equation*}
and with the corresponding \emph{partition function}
$$
Z_N=\int_\R\dots\int_\R \, \prod_{i<j} (\lambda_i-\lambda_j)^2 \exp\left( -  \sum_{i=1}^N V(\lambda_i)\right) \, d\lambda_1\dots d\lambda_N.
$$
In this sense, the \emph{free energy} of this matrix model is defined as
\begin{equation*}
F_N=-\frac{1}{N^2}\log Z_N.
\end{equation*}
In the physical context it is very important to study the (thermodynamical) limit
\begin{equation*}
F_\infty=\lim_{N\to\infty}F_N
\end{equation*}
(the so-called infinite volume free energy). The existence of this limit
has been established  under very general conditions on $V$, see e.g.~\cite{Johansson}, and is given in terms of the equilibrium measure in the external field $\varphi = \frac{V}{2}$.

It has been particularly studied the case of polynomial potentials $V$, and more precisely, the situation when $V$ is a quartic polynomial (see e.g. the recent monograph by Wang \cite{Wangbook} or the papers \cite{AMM2010}, \cite{BlEy2003}, \cite{Bleher99}, \cite{BPS95}, \cite{KuML 00} and \cite{MOR2013}, among many others), paying special attention to the phase transitions. In the current paper, rational external fields are considered. The motivation comes from, for instance, the generalized Gauss-Penner model considered in \cite{Kris2006}, a $1$-matrix model whose action is given by $\dsty V(M) = tr\,\left(M^4-\,\log (v+M^2)\right)\,,$ to get a computable toy-model for the gluon correlations in a baryon background. The dimensionless parameter $v>0$ stands for the ratio of quark mass to coupling constant, and the logarithmic term (responsible of the rational nature of the external field) encapsulates the effect of the $N$-quark baryon.

\section{Dynamics of the equilibrium measure in the generalized Gauss-Penner external field}
As it was announced above, through this section the simple but illustrative example of the so-called generalized Gauss-Penner model will be studied. Namely, we analyze the evolution of the equilibrium measure $\lambda = \lambda_t$ with support in the real axis in the presence of the external field
\begin{equation}\label{real axis}
\varphi(x) = \frac{x^4}{4} + b x^2\,+\,c\,\log(x^2+v)\,,\;b, c \in \R\,,\,v >0\,,
\end{equation}
when the total mass of the measure, $t$, grows from $0$ to $+\infty$; in addition, the variation of the equilibrium measure when the prescribed charge approaches the real axis, that is, when $v\searrow 0$, will be also studied (for $v=0$ and $c <0$, the classical Gauss-Penner setting is recovered, see \cite{Deo1993} and \cite{Deo2002}, among others).

Since the external field is an even function, this problem is equivalent to study the equilibrium measure $\mu = \mu_t$ of total measure $t$ and support in the real semi-axis $[0,+\infty)$ in the presence of the external field (see \cite[Th. 1.10]{Saff:97})
\begin{equation}\label{semiaxis}
\Psi(x) = 2\varphi (\sqrt x) = \frac{x^2}{2} + 2 b x + 2 c \log(x+v)\,,\,x\in[0,+\infty)\,.
\end{equation}
In addition, the respective densities are related by the expression
\begin{equation}\label{reldensities}
\mu'(x) = \frac{\lambda'(\sqrt{x})}{\sqrt{x}}\,.
\end{equation}
The outline of this Section is as follows. In Subsection 2.1, the evolution of the equilibrium measure when $t$ grows from $0$ to $\infty$ will be analyzed, while our concern in Subsection 2.2 will be the variation of the equilibrium measure when the position of the charge $v\searrow 0$ (that is, when this prescribed charge approaches the real axis).

\subsection{Dynamics with respect to the size of the measure}

Now, the variation of the equilibrium measure $\lambda_t$ when the size of the measure travels across $(0,+\infty)$ is handled. From results in Theorem 1.1 we know that in this case the numerator of $R$ has degree $10$ and $deg A \in \{2,4,6\}$ and, thus, $\supp \lambda_t$ consists of $1,2$ or $3$ intervals. Note that the symmetry of the external field $\varphi$ implies the symmetry of the support. As it was said above, it is equivalent to study the equilibrium problem in $[0,+\infty)$ for the external field \eqref{semiaxis}. Furthermore, in order to reduce the number of parameters of the equilibrium problem, it is useful the following technical result.
\begin{lemma}\label{lemma:simplif}
 If $\mu_1$ is the equilibrium measure of total mass $t_1$ in $\Sigma_1$, in the presence of the external field $\varphi_1$ and with density function $\mu'_1(x)=\omega_1(x)$, and $\supp \mu_1=\Delta_1$; then, for any $\delta,\kappa>0$ the measure $\mu_2$, with density $\omega_2(x)=\mu'_2(x)=\delta \kappa\omega_1(\delta  x)$, is the equilibrium measure of total mass $t_2=\kappa t_1$ for $\Sigma_2=\Sigma_1/\delta$ in the external field $\varphi_2(x)=\kappa\varphi_1(\delta x)$.
\end{lemma}
\begin{proof} The proof follows straightforward from the identities ($x=\delta y$)
\begin{align*}
\int_{\Delta_1}&\log\frac{1}{|s-x|}\omega_1(s)ds+\varphi_1(x)
=\int_{\Delta_2}\log\frac{1}{|\delta s-\delta y|}\delta \omega_1(\delta s)ds+\varphi_1(\delta y)\\
=&\int_{\Delta_2}\log\frac{1}{|s-y|}\delta \omega_1(\delta s)ds+\varphi_1(\delta y)+\text{const}
=\frac{1}{\kappa}\left(\int_{\Delta_2}\log\frac{1}{|s-y|}\delta \kappa\omega_1(\delta s)ds+\kappa\varphi_1(\delta y)+\text{const}\right)\,,
\end{align*}
  which implies that $\mu_2$, given by its density $\omega_2(x)=\mu'_2(x)= \delta \kappa \omega_1(x)$ in $\Delta_2$, is the equilibrium measure in $\Sigma_2$ in the external field $\varphi_2$, with total mass $t_2$ given by
$$t_1=\int_{\Delta_1}\omega_1(s)ds=\int_{\Delta_2} \delta \omega_1(\delta s)ds=\frac{1}{\kappa} \int_{\Delta_2} \delta \kappa\omega_1(\delta s)ds=\frac{1}{\kappa}\int_{\Delta_2}\omega_2(s)ds=\frac{1}{\kappa}t_2\,.$$

\end{proof}

As a consequence, setting $\delta=v$ and $\kappa=1/v^2$, we can study the equilibrium measure $\nu = \nu_t$ supported in the positive semi-axis in the external field $$\kappa \Psi(\delta x)=\frac{1}{ v^2}\left(\frac{(vx)^2}{2}+2 b (vx)+2 c \log(vx+v)\right)=
\frac{1}{2}x^2+\frac{2 b}{v}x+\frac{2 c}{v^2}\log(x+1)+\frac{2 c}{v^2}\log v\,,$$
or, what is the same, recalling the parameters and removing the constant term, we finally arrive to the external field
\begin{equation}\label{simplified}
\phi(x)=\frac{1}{2}x^2+\beta x+\gamma \log (x+1)\,,\,x\in[0,+\infty)\,.
\end{equation}

Hence, for simplicity, through this subsection, where dynamics with respect to $t$ is considered, we deal with the equilibrium problem in the positive semi-axis in the external field \eqref{simplified}; the conclusions could be readily translated to the original setting in the whole real axis in the external field \eqref{real axis}, taking $\beta = \frac{2b}{v}$ and $\gamma = \frac{2c}{v^2}$.

It is also noteworthy to mention that the external field \eqref{simplified} admits a simple electrostatic interpretation. It consists of a polynomial part plus the potential created by a charge located at $-1$. Coefficient $|\gamma|$ is the mass of this charge, in such a way that $\gamma>0$ means that we treat with an attractive charge, while $\gamma<0$ stands for a repulsive charge. Coefficient $\beta$ is related with whether $0$ is a local maximum (and thus, it tries to repel the support of the equilibrium measure) or a local minimum (so, it attracts the support), namely:  if $\beta+\gamma>0$, then $0$ is a local minimum and if $\beta+\gamma<0$, then $0$ is a local maximum.

Now, for the sake of transparency in our discussion, we start announcing our main results throughout this subsection. These results are written in terms of the equilibrium measure $\nu$ in the presence of the field \eqref{simplified}. In order to do it, let us set, first, the critical points of $\phi$ in \eqref{simplified}, that is,
\begin{equation}\label{criticalp}
y_{\pm} = -\frac{\beta +1}{2}\,\pm \,\sqrt{\left(\frac{\beta -1}{2}\right)^2-\gamma}\,.
\end{equation}
In addition, let us define the following open regions $\mathcal{A},\mathcal{B},\mathcal{C} \subset \R^2$ in the $(\beta,\gamma)\,$-plane, being pairwise disjoint and such that $\overline{\mathcal{A}\cup \mathcal{B}\cup \mathcal{C}}=\R^2\,.$
\begin{equation*}\label{openregions}
\begin{split}
\mathcal{A} = & \left\{(\beta,\gamma)\in R^2\,:\,\beta >-1\;\text{and}\;\gamma >-\beta\;\;\text{or}\;\;\beta <-1\;\text{and}\;\gamma >  \left(\frac{3-2\beta}{5}\right)^{5/2}\right\}\,,\\
\mathcal{B}= & \left\{(\beta,\gamma)\in R^2\,:\,\beta >-1\;\text{and}\;\gamma <-\beta\;\;\text{or}\;\;\beta <-1\;\text{and}\;\gamma < \sqrt{-1-2\beta}\right\}\,,\\
\mathcal{C} = & \left\{(\beta,\gamma)\in \R^2\,:\,\beta < -1\,\text{and} \,  \sqrt{-1-2\beta} < \gamma < \left(\frac{3-2\beta}{5}\right)^{5/2}\right\}\,.
\end{split}
\end{equation*}

For our purposes it is also important to know whether the absolute minimum of $\phi$ in $[0,+\infty)$ is $y_+$ or the origin. In this sense, take into account that $2\phi(y_+) = (\beta -1)y_+ +2 \gamma \log (y_++1)-\beta -\gamma\,,$ and we have that equation $\phi(y_+) = 0$ defines a curve in the $(\beta ,\gamma)$ plane passing through the point $(-1,1)$ and such that for $\beta <-1$, this is a monotonically decreasing curve placed between the bisector $\beta + \gamma = 0$ and the left branch of the parabola $\displaystyle \gamma = \left(\frac{\beta +1}{2}\right)^2$, that is, $\beta = 1-2\sqrt{\gamma}\,.$ It is easy to check that on the right and left-hand sides of this curve, it holds $\phi(y_+) > 0$ and $\phi(y_+) < 0$, respectively (below, these subregions will be denoted by $\mathcal{C}_+$ and $\mathcal{C}_-$, respectively, see Figure 1).

Now, we have
\begin{theorem}\label{thm:two-cut}
For the regions of the $(\beta,\gamma)\,$-plane defined above, we have

\begin{itemize}

\item [(a)] For $(\beta,\gamma)\in \mathcal{A}\,,$ $S_t = [0,a_1]\,,$ for any $t>0$, $a_1=a_1(t)$ being an increasing function of $t$.

\item [(b)] For $(\beta,\gamma)\in \mathcal{B}$, there exists $T_0=T_0(\beta,\gamma)>0$ such that:
	\begin{itemize}
	\item [(i)] For $t<T_0$, $S_t=[a_2,a_3]$ with $a_2=a_2(t)$ a decreasing function and $a_3=a_3(t)$ an increasing function. Furthermore, 	 $\lim_{t\nearrow T_0}a_2=0$.
	\item [(ii)] For $t\geq T_0$, $S_t=[0,a_1]$ with $a_1=a_1(t)$ an increasing function.
	\end{itemize}
      In addition, we have that $\lim_{t\nearrow T_0}a_3=\lim_{t\searrow T_0}a_1$.

\item [(c)] For $(\beta,\gamma)\in \mathcal{C}$, there exists $0\leq T_1=T_1(\beta,\gamma)<T_2=T_2(\beta,\gamma)<+\infty$ such that
	\begin{itemize}
	\item [(i)] For $t<T_1$, $S_t$ consists of one interval:
		\begin{itemize}
		\item If $(\beta,\gamma)\in \mathcal{C}_+$, $S_t=[0,a_1]$, with $a_1=a_1(t)$ an
            increasing function;
		\item If $(\beta,\gamma)\in {C}_-$, $S_t=[a_2,a_3]$, with $a_2=a_2(t)$ a
        decreasing function and $a_3=a_3(t)$ an increasing function. Furthermore, $\lim_{t\nearrow T_1}a_2>0$.
		\end{itemize}
    \item [(ii)] For $t\in(T_1,T_2)$, $S_t=[0,a_1]\cup [a_2,a_3]$ with $a_1=a_1(t)$ and $a_3=a_3(t)$ increasing functions while $a_2=a_2(t)$ is a decreasing function. Furthermore
        \begin{itemize}
            \item If $(\beta,\gamma)\in \mathcal{C}_+$ then $\lim_{t\nearrow T_1}a_1<\lim_{t\searrow T_1}a_2=\lim_{t\searrow T_1}a_3$.
            \item If $(\beta,\gamma)\in \mathcal{C}_-$ then $\lim_{t\searrow T_1}a_1=0$, $\lim_{t\nearrow T_1}a_2=\lim_{t\searrow T_1}a_2$ and
                $\lim_{t\nearrow T_1}a_3 =\lim_{t\searrow T_1}a_3$.
        \end{itemize}
        In addition, $\lim_{t\nearrow T_2}a_1=\lim_{t\nearrow T_2}a_2$.
    \item [(iii)] For $t>T_2$, $S_t=[0,a_1]\,,$ $a_1=a_1(t)$ being an increasing function. Moreover, $\lim_{t\nearrow T_2}a_3=\lim_{t\searrow T_2}a_1$.
    \end{itemize}

Note that $T_1=0$ if and only if the particular case where $\beta < -1$ and $(\beta,\gamma)$ belongs to the curve where $\phi(y_+)=0$ takes place.

\end{itemize}

\end{theorem}

\begin{remark}\label{ambiguity}

As it will be seen below, for the support of the equilibrium measure we have three possible scenarios, in such a way that the endpoints will be denoted according these different settings. Namely, we have

\begin{itemize}

\item If the support consists of a single interval containing the origin, then we call $a_1$ the rightmost endpoint: $S_t = [0,a_1]$.

\item If the support consists of a single interval not containing the origin, then we write: $S_t = [a_2,a_3]$.

\item Finally, If the support comprises two intervals, then necessarily an endpoint is the origin and we write $S_t = [0,a_1]\cup [a_2,a_3]$.

\end{itemize}

On the other hand, the other zeros of the density of the equilibrium measure are listed according its distance to the support (there are two of these just when the support is a single interval containing the origin). As we can see in the statement of Theorem \ref{thm:two-cut}, in certain cases this notation may seem to fall in a certain ambiguity, for instance, when two disjoint intervals merger producing a single one; nevertheless, we guess that this is a common problem to all possible conventions.

From now on, if it is not necessary, no mention of the dependence of $a_i$ and $b_j$ on the total mass $t$ will be made.

\end{remark}

\begin{remark}\label{rem:transitions}
Critical values $T_j\,,\,j=0,1,2\,,$ in Theorem \ref{thm:two-cut} can be obtained explicitly using the fact that they correspond with phase transitions, as we will show below.

On the other hand, along the upper curve delimiting $\mathcal{C}_+$ and the lower one delimiting $\mathcal{C}_-$, in Figure 1, a type III transition takes place. The other boundary in Figure 1, namely, the bisector $\beta + \gamma = 0\,,\, \beta > -1\,,$ splits the region where we have one-cut for any $t>0$ into two parts: on $\mathcal{A}$, the support $S_t$ always has the origin as an endpoint, while on $\mathcal{B}$, $S_t$ does not contain initially the origin.
\end{remark}

\begin{figure}
    \begin{center}
        \includegraphics[scale=0.3]{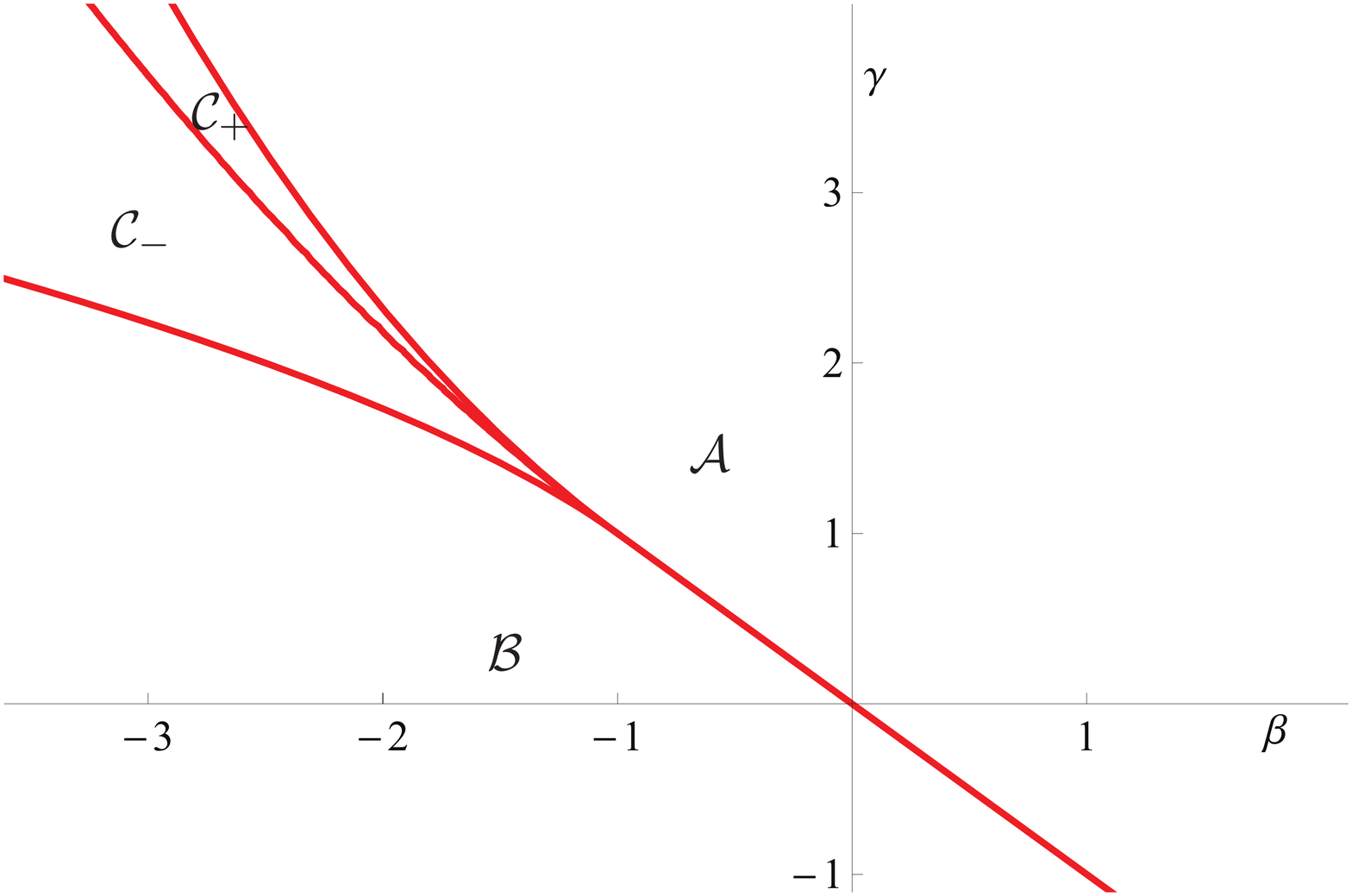}
        \caption{Different regions on $(\beta,\gamma)$-plane according to the phases of $S_t$.}
        \label{fig:mapaFases}
    \end{center}
\end{figure}

The next result deals with what happens when $t\rightarrow \infty$, that is, what we can call, roughly speaking, the ``end of the movie''. By \cite{Buyarov/Rakhmanov:99}, it is clear what is the beginning:  we know that $$\bigcap_{t>0}\,S_t = \{y\in [0,+\infty):\phi(y) = \min_{x\geq 0}\,\phi(x)\} \in \{0,\,y_+\}\,,$$
and, in fact, the set above consists of a single point except for $(\beta, \gamma)$ in the curve where $\phi(y_+)=\phi(0)=0\,.$
Now, we want to know what occurs in any case when $t$ is large enough. It will be also proved that, though there are different possible evolutions, the ``end of movie'' is very similar in all cases.

\begin{theorem}\label{thm:end}
For any $\beta \in \R$ and $\gamma \in \R\setminus \{0\}$, there exists $T = T(\beta,\gamma)>0$ such that for $t>T$, the density of the equilibrium measure has the form:
$$\nu'_t(x) = \frac{(x-b_1)(x-b_2)}{x+1}\,\sqrt{\frac{a_1-x}{x}}\,dx\,,$$ where $b_2<b_1\,,$ with $b_2 \in (-\infty,-1)\,,\,b_1 \in (-\infty,0)$, for $t>0$ and $$\lim_{t\rightarrow \infty}\,a_1(t) = +\infty\,,\;\lim_{t\rightarrow \infty}\,b_2(t) = -\infty\,,\;\lim_{t\rightarrow \infty}\,b_1(t) = -1\,.$$
\end{theorem}

Note that both in regions $\mathcal{A}$ and $\mathcal{B}$ above, we have one-cut for any $t>0$. However, since our primary aim is studying the evolution of $\lambda_t$ in terms of the external field \eqref{real axis}, we also need to know when $S_t$ consists of a single interval with the origin as an endpoint for any $t>0$. In this sense, as an immediate consequence of Theorems \ref{thm:two-cut}-\ref{thm:end}, the situation in the original setting, where the equilibrium problem in the whole real axis in the external field \eqref{quarticperturblog} is considered, may be depicted (see Figure 1 for the different scenarios). Recall that $\beta = \frac{2b}{v}$ and $\gamma = \frac{2c}{v^2}$.

\begin{corollary}\label{cor:phasediag}
We have the following ``phase-diagrams'' (that is, the evolution of the numbers of cuts when $t$ grows) for $\supp \lambda_t\,,$ the equilibrium measure in the presence of the external field \eqref{real axis} in the whole real axis:

\begin{itemize}

\item For $(\beta,\gamma) \in \mathcal{A}\,$: \textbf{one-cut}, for any $t>0$.
\item For $(\beta,\gamma) \in \mathcal{B}\,$: \textbf{two-cut}, for $\,0<t<T_0\;$  $\,\longrightarrow\,$ \textbf{one-cut}, for $\,t\geq T_0$.
\item For $(\beta,\gamma) \in \mathcal{C}_+\,$: \textbf{one-cut}, for $\,0<t\leq T_1\;$ $\,\longrightarrow\,$ \textbf{three-cut}, for $\,T_1<t<T_2\;$ $\,\longrightarrow\,$ \textbf{one-cut}, for $\,t\geq T_2$.
\item For $(\beta,\gamma) \in \mathcal{C}_-\,$: \textbf{two-cut}, for $\,0<t\leq T_1\;$ $\,\longrightarrow\,$ \textbf{three-cut}, for $\,T_1<t<T_2\;$ $\,\longrightarrow\,$ \textbf{one-cut}, for $\,t\geq T_2$.

\end{itemize}

\end{corollary}

\begin{remark}\label{rem:Krish}
At this point, it is interesting to compare our results with the analysis carried out in \cite{Kris2006}. There, the author studies for which values of parameter $v$ (the ``coupling'' constant, using his terminology) the (probabilistic) equilibrium measure (in the whole real axis) in the external field $\varphi(x) = \frac{1}{2}\,x^4 - \frac{1}{2}\,\log (x^2+v)\,$ has a one-cut or a two-cut support. He obtains a critical value for parameter $v$, namely, $v_c \approx 0.27$, such that for $v < v_c$, the support of the equilibrium measure is comprised of two disjoint intervals. Relations \eqref{real axis}-\eqref{reldensities} and Lemma \ref{lemma:simplif} imply that this problem is equivalent to set $\beta = 0$, $\gamma = -\,\frac{1}{2v^2}$ and $t = \frac{1}{2v^2}$ in our scenario (b) above, determining the critical value $v_c$ of $v$ for which $a_2 = 0$. Straightforward calculus yield the unique admissible solution $v_c = 0.269593\ldots$  which agrees with the approximate value $v_c \approx 0.27\,,$ obtained in \cite{Kris2006}.

\end{remark}

The rest of this subsection will be devoted to the proof of these main results and other complementary ones. As announced, the main ingredients for this study are the representation of the Cauchy transform of $\nu_t$, derived from Theorem 1.1 above, and the dynamical system \eqref{variationzeros} contained in Theorem 1.2.

With respect to the first of these tools, from Theorem 1.1, relations \eqref{semiaxis}-\eqref{reldensities} and Lemma \ref{lemma:simplif}, we have that for the equilibrium measure $\nu_t$ in $[0,+\infty)$ in the presence of the external field \eqref {simplified} the following identity holds:
\begin{equation}\label{Rsimplified}
(-\widehat{\nu_t} + \phi')^2(z) = R_t(z) = \frac{A(z)B(z)^2}{q(z)(z+1)^2}\,,
\end{equation}
where the zeros of $A(z)$, that is, the zeros of odd order of $R_t$, are the endpoints of the support $S_t$ not being the origin; and $q(z) = z$ or $q(z) = 1$, according to whether the origin is or is not an endpoint of the support, respectively, and, in case of being and endpoint, depending on its nature (hard or soft-edge, as it will be explained below). We also have that $deg (A) \in \{1,2,3\}$, according to the support consists of a single interval with the origin as an endpoint, a single interval with strictly positive endpoints or two disjoint intervals (necessarily, with the origin as an endpoint), respectively. In addition, $deg (A) + 2\, deg (B) - deg (q) = 4$ and the zeros of $B$ are the other zeros of the density of $\nu'$ not being endpoints of $S_t$.

Now, from \eqref{totalpotential} and \eqref{Rsimplified} we have that the possible expressions for the density of the equilibrium measure $\nu_t$ in the external field \eqref{simplified}, except for the critical values of $t$ where phase transitions take place, are the following:

\begin{itemize}

\item \textbf{(a) one-cut with the origin as an endpoint}
\begin{equation}\label{scen(a)}
\nu'_t(x) = \frac{1}{\pi }\,\frac{(x-b_1)(x-b_2)}{x+1}\,\sqrt{\frac{a_1-x}{x}}\,,\;x\in (0,a_1)\,,\:0<a_1=a_1(t)\,,
\end{equation}
with

\textbf{(a1)} $a_1=a_1(t)<b_1=b_1(t)<b_2=b_2(t)$ and $$\int_{a_1}^{b_2}\,\frac{(x-b_1)(x-b_2)}{x+1}\,\sqrt{\frac{x-a_1}{x}}\,dx \,>\,0\,;$$ or $b_2=b_2(t)<b_1=b_1(t)<0<a_1=a_1(t)$.

or

\textbf{(a2)} $b_1,b_2 \notin \R\,,\; b_2 = \overline{b_1}.$

\item \textbf{(b) one-cut without the origin as an endpoint}
\begin{equation*}\label{scen(b)}
\nu'_t(x) = \frac{1}{\pi }\,\frac{x-b_1}{x+1}\,\sqrt{(x-a_2)(a_3-x)}\,,\;x\in (a_2,a_3)\,,\;b_1=b_1(t)<a_2=a_2(t)<a_3=a_3(t)\,.
\end{equation*}
 When $0<b_1<a_2$, we have, in addition, that $$\int_0^{a_2}\,\frac{x-b_1}{x+1}\,\sqrt{(x-a_2)(x-a_3)}\,dx\,>\,0\,.$$

\item \textbf{(c) two-cut, necessarily with the origin as an endpoint}
\begin{equation}\label{scen(c)}
\nu'_t(x) = \frac{1}{\pi }\,\frac{x-b_1}{x+1}\,\sqrt{\frac{(x-a_1)(x-a_2)(a_3-x)}{x}}\,,\;x\in (0,a_1) \cup (a_2,a_3)\,,
\end{equation}
with $0<a_1=a_1(t)<b_1=b_1(t)<a_2=a_2(t)<a_3=a_3(t)\,,$ where
\begin{equation}\label{eqcond}
\int_{a_1}^{a_2}\,\frac{x-b_1}{x+1}\,\sqrt{\frac{(x-a_1)(x-a_2)(x-a_3)}{x}}\,dx\,=\,0\,.
\end{equation}
\end{itemize}

For the study of the possible scenarios above, identity \eqref{Rsimplified} provides some interesting relations for the zeros of the density of the equilibrium measure in terms of parameters $\beta, \gamma$ and the total mass $t$. Thus, equating coefficients in \eqref{Rsimplified} and taking into account the residues of $\sqrt{R_t}$ at $z=-1$, we obtain different nonlinear systems of equations for the different scenarios.

It is also easy to check the location of critical points $y_{\pm}$, introduced in \eqref{criticalp}, for the different values of $(\beta,\gamma)$. The importance of these points for our analysis relies on the fact that they provide the initial conditions for dynamical systems \eqref{variationzeros} in Theorem 1.2. Indeed, we have
\begin{equation}\label{complexy}
y_- = \overline{y_+}\in \C \setminus \R \Longleftrightarrow \gamma >0\;\text{and}\;1-2\sqrt{\gamma}<\beta <1+2\sqrt{\gamma}\,,
\end{equation}
\begin{equation}\label{locy+}
y_+ \begin{cases} > 0\,,\; & \text{if}\;\gamma >1 \;\text{and}\;\;\beta <1-2\sqrt{\gamma}\;\;\text{or}\;\gamma <1\;\text{and}\;\beta <-\gamma\,,\\
< -1\,,\; & \text{if}\;\gamma >0\; \text{and}\;\beta >1+2\sqrt{\gamma}\,,\\
\in (-1,0)\,,\; & \text{otherwise}\,, \end{cases}
\end{equation}
\begin{equation}\label{locy-}
y_- \begin{cases} > 0\,,\; & \text{if}\;\gamma >1 \;\text{and}\;-\gamma <\beta <1-2\sqrt{\gamma}\,,\\
< -1\,,\; & \text{if}\;\gamma <0 \;\;\text{or}\;\gamma >0\;\text{and}\;\beta >1+2\sqrt{\gamma}\,, \\
\in (-1,0)\,,\; & \text{otherwise}\,. \end{cases}
\end{equation}
In Fig. \ref{fig:mapays} the above possibilities for the location of $y_{\pm}$ are displayed. Namely, from \eqref{complexy}-\eqref{locy-} we have that $y_+ >0$ in regions I, II and III, $y_+ \in (-1,0)$ in IV and V, and $y_+ \in (-\infty,-1)$ in VI; while $y_- >0$ only in region I, $y_- \in (-1,0)$ in II and IV and $y_- <-1$ in III, V and VI. Finally, both $y_{\pm}$ are imaginary, with $y_- = \overline{y_+}\,,$ in region VII.
\begin{figure}
    \begin{center}
        \includegraphics[scale=0.8]{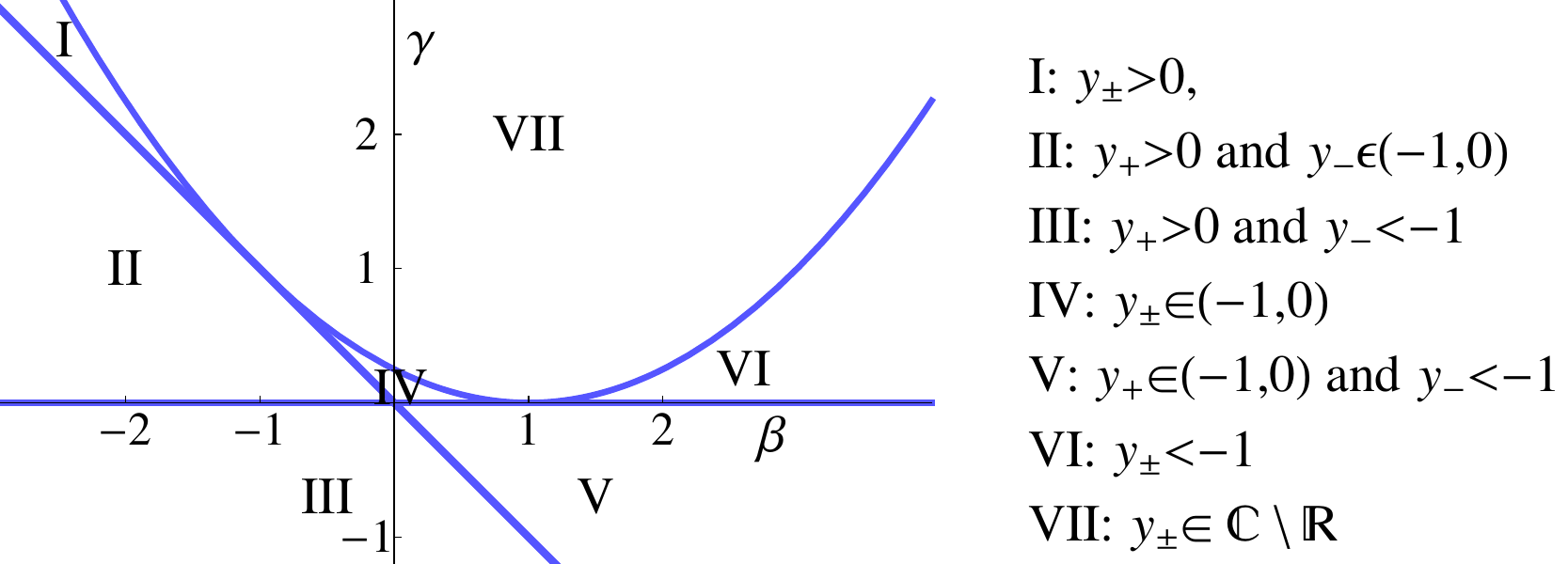}
        \caption{Location of $y_{\pm}$. Different regions on $(\beta,\gamma)$-plane.}
        \label{fig:mapays}
    \end{center}
\end{figure}

On the other hand, when $y_->0$, that is, when $\gamma >1 \;\text{and}\;-\gamma <\beta <1-2\sqrt{\gamma}\,,$ it is also interesting to know whether $\phi(y_+)>\phi(0)=0$, in order to know where $\phi$ attains its absolute minimum in $[0,+\infty)$. This question is solved in terms of the curve in the $(\beta,\gamma)\,$-plane where $\phi(y_+)=0$ described above, which splits region I into two parts and, consequently, region $\mathcal{C}$ into two parts $\mathcal{C}_{\pm}$.

It is well-known that the convexity of the external field guarantees the convexity of the support (see e.g. \cite{Saff:97}), and then, we immediately get
\begin{lemma}\label{lem:convex}
If $\gamma \leq 1$, $S_t = \supp \nu_t$ consists of a single interval for any $t>0\,.$
\end{lemma}
Moreover, in \cite[Th.1.1]{Saff:97}, other sufficient condition for the convexity of the support in $[0,+\infty)$ can be found, namely, when the function $x\phi'(x)$ is increasing in $[0,+\infty)$. From \eqref{simplified}, we have that $(x\phi')'(x) = 2x + \beta  + \frac{\gamma}{(x+1)^2}\,.$
If $\beta >-1$ and $\beta + \gamma >0$, the polynomial $q(x)=(2x+\beta)(x+1)^2+\gamma = 2x^3 + (\beta +4)x^2 + 2(\beta +1)x + \beta + \gamma$ is positive for $x>0$ and, then, $x\phi'(x)$ is an increasing function on $[0,+\infty)$. Thus, joining this fact with the result in Lemma \ref{lem:convex} we can state

\begin{lemma}\label{thm:single}
If $\gamma \leq 1$ or $\beta \geq -1$, $S_t$ consists of a single interval for any $t\geq 0$.
\end{lemma}
Therefore, our main interest is studying the open quadrant $\{(\beta,\gamma)\in \R^2: \beta <-1, \gamma >1\}\,,$ where in principle the support can consist of two intervals for some values of $t$; but it is not our unique interest. Bearing in mind the application of our results to the original problem in the whole real axis, we are also very concerned with the question about whether the origin belongs to the support (if it is not the case, the support of $\lambda_t$ cannot be a single interval).

As we said above, the other crucial ingredient for the study of the evolution of the equilibrium measure as the total mass $t$ increases from $0$ to $+\infty$ is the dynamical system for the endpoints of the support and the other zeros of the density $\nu'_t$ given in \eqref{variationzeros}. Indeed, assume that for some $t>0$ one of the scenarios \textbf{(a)-(c)} above takes place. Then, one has,

\begin{itemize}

\item \textbf{(a1)}

\begin{equation}\label{a1}
\begin{split}
\dot{a_1} & = \frac{2(a_1+1)}{(a_1-b_1)(a_1-b_2)}\,,\;\dot{b_1} = \frac{b_1+1}{(b_1-a_1)(b_1-b_2)}\,, \\[.3cm]
\dot{b_2} & = \frac{b_2+1}{(b_2-a_1)(b_2-b_1)}\,,
\end{split}
\end{equation}
where
\begin{equation}\label{a1signum}
\begin{cases}
\dot{a_1}>0, \dot{b_1}<0, \dot{b_2}>0\,,\; & \text{if}\;0<a_1<b_1<b_2\,, \\
\dot{a_1} >0, \dot{b_1} <0, \dot{b_2} >0\,,\; & \text{if}\;-1<b_2<b_1<0<a_1\,, \\
\dot{a_1} >0, \dot{b_1} >0, \dot{b_2} <0\,,\; & \text{if}\;b_2<b_1<-1<0<a_1\,, \\
\dot{a_1} >0, \dot{b_1} <0, \dot{b_2} <0\,,\; & \text{if}\;b_2<-1<b_1<0<a_1\,.
\end{cases}
\end{equation}

\item \textbf{(a2)}
\begin{equation}\label{a2}
\begin{split}
\dot{a_1} & = \frac{2(a_1+1)}{|a_1-b_1|^2}\,>\,0\,,\;\dot{(\Re b_1)}\,=\,-\frac{a_1+1}{2|a_1-b_1|^2}\,<\,0\,, \\[.3cm]
\dot{(\Im b_1)}\,& =\,-\frac{|a_1-b_1|^2+(a_1+1)(\Re b_1 -a_1)}{2\Im b_1 \,|a_1-b_1|^2}\,.
\end{split}
\end{equation}

\item \textbf{(b)}
\begin{equation}\label{b}
\dot{a_2} = \frac{2(a_2+1)}{(a_2-a_3)(a_2-b_1)}\,<\,0\,,\;\dot{a_3} = \frac{2(a_3+1)}{(a_3-a_2)(a_3-b_1)}\,>\,0\,,
\end{equation}
and
\begin{equation*}\label{bbis}
\dot{b_1}  = \frac{b_1+1}{(b_1-a_2)(b_1-a_3)}\,\begin{cases} >\,0\,,\; & \text{if}\; b_1 >-1\,, \\
<\,0\,,\; & \text{if}\; b_1 <-1\,, \end{cases}
\end{equation*}
where taking residues at $z=-1$ in $\sqrt{R}$ (see \eqref{Rsimplified}) shows that $b_1 > -1$ if and only if $\gamma > 0$.

\item \textbf{(c)}
\begin{equation}\label{c}
\begin{split}
\dot{a_1} & = \frac{2(a_1+1)(a_1-\xi)}{(a_1-b_1)(a_1-a_2)(a_1-a_3)}\,>\,0\,,\;\dot{b_1} = \frac{2(b_1+1)(b_1-\xi)}{(b_1-a_1)(b_1-a_2)(b_1-a_3)}\,, \\[.3cm] \dot{a_2} & = \frac{2(a_2+1)(a_2-\xi)}{(a_2-b_1)(a_2-a_1)(a_2-a_3)}\,<\,0\,,\;\dot{a_3} = \frac{2(a_3+1)(a_3-\xi)}{(a_3-b_1)(a_3-a_2)(a_3-a_1)}\,>\,0\,,
\end{split}
\end{equation}

where $b_1,\xi \in (a_1,a_2)$ satisfy, respectively:
    $$\int_{a_1}^{a_2}\,\frac{x-b_1}{x+1}\,\sqrt{\frac{(x-a_1)(x-a_2)(x-a_3)}{x}}\,dx\,=\,0\,,\;\int_{a_1}^{a_2}\,\frac{x-\xi}{\sqrt{x(x-a_1)(x-a_2)(x-a_3)}}\,dx\,=\,0\,.$$

\end{itemize}

Now, we are in a position of proving the main results.

\textbf{Proof of Theorem \ref{thm:two-cut}}

First, we concentrate in the region of the $(\beta,\gamma)\,$-plane where $\beta <-1$ and $\gamma > 1$, for which a two-cut phase is, in principle, feasible.
Observe that this domain consists of three different regions concerning our previous analysis  on the values of $y_{\pm}$ in \eqref{complexy}-\eqref{locy-}.

First, for $\dsty -\beta < \gamma < \left(\frac{\beta -1}{2}\right)^2\,,$ (region I in Figure 2) we have that both $y_{\pm}>0$. This region is split by the curve where $\phi(y_+)=0$,
in such a way that in the right side of this region, that is, when $\phi(y_+)>0$, the absolute minimum of $\phi$ in $[0,+\infty)$ is attained at the origin and, thus, for $t$ small enough,
$S_t = [0,a_1(t)]\,,$ with $a_1(t)$ an increasing function. The scenario (a1) takes place, with $a_1(0)=0 < b_1(0) = y_- < y_+ = b_2(0)$. From \eqref{a1signum} we see that
$a_1$ and $b_1$ are going to collide in a finite time $T_0>0$. Indeed, if it were possible that $a_1$ and $b_1$ do not collide at a finite time, it would imply that $\dot{a_1}$ and $\dot{b_1}$ tend to $0$ as $t\rightarrow +\infty$, but from \eqref{a1} this would also imply that $b_2$ diverges to $+\infty$. But taking residues at infinity in \eqref{Rsimplified}, with the density $\nu'_t$ given by \eqref{scen(a)}, yields that the (weighted) average of zeros of $A$ and $B$ is constant, that is, $a_1+2(b_1+b_2) = -2(\beta +1)\,$ and, thus, it is not possible that just one of these zeros diverges.

Therefore, it is clear that $a_1$ and $b_1$ are in course of collision. But if that collision actually took place,
it is easy to see that for the collision time $t=T_0$, $$\int_{a_1(T_0)}^{b_2(T_0)}\,\frac{(x-b_1(T_0))(x-b_2(T_0))}{x+1}\,\sqrt{\frac{x-a_1(T_0)}{x}}\,dx < 0\,,$$ and this would be in contradiction with \eqref{equilibrium}.
Since the function $$F(t) = \int_{a_1(t)}^{b_2(t)}\,\frac{(x-b_1(t))(x-b_2(t))}{x+1}\,\sqrt{\frac{x-a_1(t)}{x}}\,dx$$ is strictly monotonically decreasing (it is easy to check that $F'(t)<0$) and $F(0)>0$,
we have that there exists a unique $0<T_1<T_0$ such that $F(T_1) = 0$. Therefore, at $t=T_1$, the initial configuration must change, and the unique possible change
in order to satisfy \eqref{equilibrium} is the bifurcation of $b_2$ in two points $a_1,a_2$, which become the endpoints of a new cut.
Thus, at $t=T_1$ we have the birth of a new cut or, in other words, a type I singularity.

Then, for $T_1<t<T_1+\varepsilon$, we are in the scenario (c), that is, the two-cut situation. Taking into account the dynamical system \eqref{c}, it is clear that endpoints $a_1$ and $a_2$
(and, thus, also $b_1$) finally collide at a finite time, $T_2 > T_1$. At $t=T_2$, a type II singularity takes place: the merger (or fusion) of two cuts or, what is the same,
the closing of a gap. At this point, a quadruple root of $R_t$ is born from the collision of $a_1, a_2$ and $b_1$ (which is a double root). Then, for $t>T_2$, the one-cut situation follows unchanged.

A similar conclusion follows for the region where $y_{\pm}>0$ but $\phi(y_+)<0$. In this case, we start from scenario (b) above, with $0<b_1<a_2<a_3$ and
\begin{equation}\label{compat}
\int_0^{a_2}\,\frac{x-b_1}{x+1}\,\sqrt{(x-a_2)(x-a_3)}\,dx\,>\,0\,.
\end{equation}
Initially, $a_2(0)=a_3(0)=y_+>y_-=b_1(0)>0\,.$ Similar arguments as above shows that there exists a finite time $T_0>0$ such that $b_1(T_0)=a_2(T_0)$ but, at that moment $T_0$ condition \eqref{compat} would be violated. As above, there is a unique admissible change in the initial configuration taking care of \eqref{equilibrium}: this is the appearance of a new cut $[0,a_1]$, starting form the origin, at a certain time $0<T_1<T_0$, that is, $a_1(T_1)=0$. Thus, immediately after $t=T_1$, we have a two-cut support: $[0,a_1(t)]\cup [a_2(t),a_3(t)]$, satisfying the dynamical system \eqref{c}. Then, there exists $T_2>T_1$ such that $a_1(T_2)=b_1(T_2)=a_2(T_2)$ and the merger of the two cuts occurs, as above. Therefore, we have that in the region where $y_{\pm}>0$, that is, when $\gamma >1$ and $-\gamma < \beta < 1-2\sqrt{\gamma}$, two-cut situation necessarily takes place. In these cases the dynamics of the equilibrium measure presents both type I (birth of a cut) and type II (merger of cuts) singularities.

Now, we are concerned with the region given by $\gamma >1$ and $1-2\sqrt{\gamma}<\beta<-1$, where $y_- = \overline{y_+} \notin \R$ (contained in region VII in Figure 2). Thus, we start from scenario (a2) and the dynamical system \eqref{a2}: the right endpoint of the support, $a_1(t)$, is increasing with $t$ and the real part of $b_1$ is decreasing but, what about the imaginary part? Taking a look to the last expression in \eqref{a2}, we see that for each $t>0$, the support $[0,a_1(t)]$ is contained in the circle
\begin{equation}\label{circleDt}
D_t\,=\,\left\{z\in \C\,:\,\left|z-\frac{a_1-1}{2}\right|\leq \frac{a_1+1}{2}\right\}\,,
\end{equation}
in such a way that points $a_1(t)$ and $-1$ are the endpoints of the real diameter of the disk. Now, we can see that if $b_1(t)$ lies outside $D_t$, then its imaginary part decreases; on the opposite, if $b_1(t)\in D_t$, then $\Im b_1$ increases. The boundary between these two behaviors is the case where $b_1$ (and $\overline{b_1}$) and $a_1$ collide at a certain time $t=T$. This is a type III singularity, and a quintuple root of $R_T$ occurs in $a_1=a_1(T)=b_1(T)=\overline{b_1(T)}$. Indeed, equating coefficients in \eqref{Rsimplified} and taking into account the residue at $z=-1$ of $\sqrt{R_T(z)}$, we obtain the following system of equations at $t=T$: $$2(\beta +1) = -5a_1\,,\, (\beta +1)^2+2(\gamma + \beta -T) = 10a_1^2\,,\,\gamma = (1+a_1)^{5/2}$$ and, thus, $$\gamma = \left(\frac{3-2\beta}{5}\right)^{5/2}\,.$$ Moreover, $a_1(T) = -\frac{2(\beta +1)}{5}$. It is easy to see that for $\gamma > \left(\frac{3-2\beta}{5}\right)^{5/2}\,,$ $b_1$ and $\overline{b_1}$ enter in the circle $D_t$ before they collide in the real axis; after that, $\Im b_1$ increases and, thus, the one-cut setting follows for any $t>T$, since in case that $b_1$ finally leaves the disk and $\Im b_1$ decreases again, $b_1$ and $\overline{b_1}$ only could collide in the real semi-axis $(-\infty,-1)$ (because of $\Re b_1$ is always decreasing) and a new cut cannot be born. On the contrary, when $\gamma < \left(\frac{3-2\beta}{5}\right)^{5/2}\,,$ $b_1$ and $\overline{b_1}$ attains the real axis forming a quadruple real root of $R_{T_0}$: $b_1 = b_1(T_0) \in \C \setminus D_t$ . Thus, for $t>T_0$, this quadruple root splits into two double real roots (is the unique feasible option) and the phase diagram drawn above, that is, the birth of a new cut at the rightmost double root (at a certain time $T_1>T_0$) and the further merger of two cuts at $T_2>T_1$, takes place again.

Now, it remains to consider the region where $\gamma >1$ and $\beta <-\gamma$ (contained in region II in Figure 2). In this case, for small $t$ we are in the scenario (b), and the dynamical system \eqref{b} holds with $-1<y_-<0=b_1(0)<0<a_2(0)=a_3(0)=y_+$. Thus, $\dot{b_1}>0,\dot{a_2}<0,\dot{a_3}>0$ and, therefore, both $b_1$ and $a_2$ tend to $0$. The question is: who attains before the origin? It is clear that if $a_2$ is the ``winner'', that is, if there exists a time $T_0>0$ for which $a_2(T_0)=0$ and $b_1(T_0)<0$, then for $t>T_0$ $b_1$ goes far away from the origin and the support will be $S_t = [0,a_3(t)]$, with $a_3(T)$ an increasing function for $t\geq T_0$ (to be coherent with our notation, from this moment, $a_3$ will be renamed $a_1$). On the contrary, if for some $T_0>0$, we have that $b_1(T_0)=0$ and $a_2(T_0)>0$, then $b_1$ will enter into the positive semi-axis and $b_1$ and $a_2$ will be in course of collision; but this collision cannot occur since in this case \eqref{equilibrium} would be violated. Thus, necessarily at a certain $T_1>T_0$, previous to the collision time, the origin creates a new cut $[0,a_1(t)]$ (birth of a cut), and the two-cut phase starts. As above, for some value $T_2>T_1$, the fusion of cuts must take place. The boundary case occurs precisely when $b_1$ and $a_2$ attains the origin at the same time $T>0$. In that case, $R_T$ has a triple root at the origin and, from \eqref{Rsimplified} and taking into account the residue at $z=-1$ of $\sqrt{R_T(z)}$, we obtain the following system of equations at $t=T$:
$$2(\beta +1) = -a_3\,,\, (\beta +1)^2+2(\gamma + \beta -T) = 0\,,\,\gamma = \sqrt{1+a_3}\,.$$ Thus, we have that $\gamma = \sqrt{-1-2\beta}$ and
$T = \frac{(\beta +1)^2}{2}\,+\beta +\sqrt{-1-2\beta}>0\,.$

The proof of parts (a) and (b) of Theorem 2.1 is easy taking into account the discussion above and the fact that for $(\beta,\gamma)\in A$, we have that $\beta+\gamma>0$ and, thus, the origin is a local minimum of $\phi$; conversely, for $(\beta,\gamma)\in B$, $\beta+\gamma<0$ holds and the origin is a local maximum of the external field. Hence, the proof of Theorem \ref{thm:two-cut} is completed.

\begin{remark}\label{rem:circle}
Through the proof of Theorem \ref{thm:two-cut}, all the singularities considered above (types I-III) have appeared. Type I and II singularities are associated to changes in the number of cuts of the support. However, maybe the most interesting is the type III one, where no topological change takes place but an analytical ``catastrophe'' occurs: the derivative of the Robin constant of the support $S_t$ has an infinite jump there or, in the random matrix framework, the limit free energy has a third order phase transition with an infinite jump (see \cite[Section 4.3]{MOR2013}). This transition marks one of the boundaries between the region of the $(\beta,\gamma)$--plane where it is possible a two-cut support and the region where it is not possible. In \cite{MOR2013}, such a difference was given in terms of which side of a certain equilateral hyperbola the pair of conjugate imaginary zeros of $B_t$ lie on; in the present case, this boundary role is played by the circle $D_t$ defined in \eqref{circleDt}.

\end{remark}

\begin{remark}\label{rem:newphasetr}
On the other hand, in the proof of Theorem \ref{thm:two-cut} it also seems to appear some new types of singularities, namely:

\begin{itemize}

\item[(i)] In scenario (b), when the left endpoint $a_2$ attains the origin, as a soft-edge, which immediately becomes a hard-edge of the support.

\item[(ii)] The birth ``from nothing'' of a new cut with the origin as a hard edge.

\item[(iii)] When the double root of $R$, $b_1$ and the left endpoint $a_2$ (simple root of $R$) collide at the origin, producing a triple root of $R$.

\end{itemize}

However, they are actually not new types of singularities if we see them from the viewpoint of the original equilibrium problem of $\lambda_t$ in $\R$, to avoid the ``boundary effect'' of the origin in the simplified model. To explain it, note that what happens in the context of the equilibrium problem of $\lambda_t$ is that the numerator of $R$ is a polynomial of degree $6$, and it is possible that different types of singularities occur simultaneously (as it was pointed out above). Thus, from the viewpoint of the original problem in the real axis, in (i) a type II singularity (closing of a gap) occurs; (ii) corresponds to a type I singularity (birth of a cut around the origin). Finally, (iii) is maybe the most interesting: it takes place simultaneously the collision of conjugate imaginary zeros of $R$ at the origin and the closing of a gap a the same point; in other words, a combination of a type III and a type II singularity occurs at the origin (the symmetry of the original field $\varphi$ obviously reinforces the role of this point).

On the other hand, as we said in the Introduction, the simplified problem in $[0,+\infty)$ we are dealing with is also of interest itself and it has not to be interpreted as coming from a symmetric problem in the real axis. In this last sense, we think that behaviors described in (i)-(iii) above are worthwhile to analyze a bit more. So, when (i) occurs the left endpoint of the support attains the origin as a soft-edge (i.e., with a ``square root-type'' behavior) and, immediately, the origin turns into a hard-edge (``inverse of square root-type'' behavior) due to the nature of this point as a boundary of the conductor. Maybe, more interesting is the situation in (ii): there, $z=0$, a regular point of $R$, gives birth simultaneously a pole and a zero of $R$ (something like a ``dipole'', in a physical sense); the pole means a hard-edge of the new cut of the support at the origin and the zero, the right endpoint of this new cut, which immediately goes far from the origin. Finally, (iii) occurs when (i) and (ii) take place simultaneously.

\end{remark}

\textbf{Proof of Theorem \ref{thm:end}}

From the discussion above it is clear that there exists $T^*>0$ such that $S_t = [0,a_1(t)]$ for $t>T^*\,,$ where $a_1(t)$ is an increasing function of $t$. Thus, it remains to prove the evolution of the other two (double) zeros of the polynomial $R_t$. We assume that for $t>T^*\,,$ these two remainder zeros, say $b_1,b_2$ are negative real numbers $b_2<b_1<0$ or a pair of conjugate imaginary numbers $b_2=\overline{b_1} \in \C \setminus \R\,,$ since by Theorem \ref{thm:two-cut} we know that the other option, i.e. $0<a_1<b_1<b_2$, produces first the birth of a new cut and, later, the closing of the gap, in such a way that finally the pair of double zeros become imaginary.

Suppose, first, that $b_2<b_1<0$. Then, taking into account the dynamical systems for the different scenarios, we have:
\begin{itemize}

\item If $b_2<b_1<-1\,,$ then for $t>T^*$, $\dot{b_2}<0$ and $\dot{b_1}>0$ and, thus, $b_2\rightarrow -\infty$ and $b_1\rightarrow -1^-$.

\item If $b_2<-1<b_1<0$, it is also easy to see that it holds $b_2\rightarrow -\infty$ and $b_1\rightarrow -1^+$.

\item If $-1<b_2<b_1<0$, then $\dot{b_2}<0$ and $\dot{b_1}<0$ and, therefore, $b_2$ and $b_1$ collide in a finite time $T^{**}$ producing a quadruple root. Then, this quadruple root splits into two conjugate imaginary double roots. Thus, this case becomes the case where $b_2=\overline{b_1} \in \C \setminus \R\,.$

\end{itemize}

Therefore, it only remains to deal with the situation where $b_2=\overline{b_1} \in \C \setminus \R\,.$ In addition, it may be assumed that $\Re b_1 >-1\,,$ because otherwise, since $\displaystyle \frac{\partial \Re b_1}{\partial t} < 0\,,$ $b_1$ and $\overline{b_1}$ would be in course of collision in the real axis, producing a quadruple root in $(-\infty ,-1)$ and we would enter in the first scenario above. Indeed, if for $t=T^*$, $\Re b_1 > a_1$, we have that $b_1$ will enter into the disk $D_t$ for some $t=T^{**}>T^*$. Then, since $\dot{(\Re b_1)} < 0\,,$ it finally ``escapes'' from $D_t$ by the left-hand side and $b_1$ and $\overline{b_1}$ will collide in the real semi-axis $(-\infty,-1)$, like in the first case above (take into account that before the final escape from $D_t$, it is possible in principle that $b_1$ escapes and be captured to escape again several times).

This renders the proof.

\begin{remark}\label{rem:illustration}
It seems convenient to show with detail by means of an illustrative example, how a combined use of the nonlinear system of equations derived from \eqref{Rsimplified}, by taking residues at $\infty$ and at $-1$, together with the dynamical systems \eqref{a1}-\eqref{c}, can be used to give a detailed description for the evolution of the main parameters of the equilibrium measure. With this purpose in mind, set $\beta=-4$ and $\gamma=3$, for which $y_-\approx -0.3$ while $y_+\approx 3.3$ (region II in Fig. 2 and region $\mathcal{C}_-$ in Fig. 1). This means that we are initially in scenario $(b)$; $a_2$ and $a_3$ are born at $y_+$, while $b_1$ begins at $y_-$. Equations \eqref{b} assert that $a_3$ increases while $a_2$ and $b_1$ approach the origin. From \eqref{Rsimplified}, the values of these roots can be obtained by solving the nonlinear system of equations
\begin{align*}
\begin{cases}
a_2+a_3+2b_1+1= & -2\beta,\\
(a_2^2-a_3)^2-(8+4a_3+4a_2)(1+b_1)= & 8(t-\gamma),\\
\sqrt{(1+a_2)(1+a_3)}(1+b_1) = & \gamma.
\end{cases}
\end{align*}

The reader can check that $b_1$ is the ``winner'' in the race against $a_2$ to attain the origin; this moment corresponds with $t=5/2$, with $a_2=3-\sqrt{7}$ and $a_3=3+\sqrt{7}$. As $t$ increases, $b_1$ becomes positive and continues increasing. At a certain moment, the saturation of inequality \eqref{compat} occurs; the corresponding $t$ can be obtained solving the equality in \eqref{compat} and can be estimated numerically: $b_1\approx 0.07$, $a_2\approx 0.17$, $a_3\approx 5.68$ and $t\approx 2.58$. Thus, we are dealing with a type I phase transition (birth of a cut at the origin).

The next phase, in which scenario (c) (two-cut) takes place, begins with a decomposition of the regular point $z=0$ (where there was not a zero or a pole of $R$) in a zero of $A$ and a pole at the hard-edge $z=0$. Then, as $t$ increases, from \eqref{c} it follows that $a_1$ also increases, while $a_2$ decreases, so they tend to collide; in addition, $a_3$ increases. The approximated values can be obtained solving numerically the equations
            \begin{align*}
            \begin{cases}
                a_1+a_2+a_3+2b_1+2 = & -2\beta,\\
                a_1^2+a_2^2+a_3^2-2a_1a_2-2a_1a_3-2a_2a_3-4(b_1+1)(a_1+a_2+a_3)-8-8b_1
                    = & -8(\gamma-t),\\
                \sqrt{(1+a_1)(1+a_2)(1+a_3)}(b_1+1) = & \gamma,\\
                \int_{a_1}^{a_2}\frac{\sqrt{(x-a_1)(a_2-x)(a_3-x)}(x+b_1)}{\sqrt{x}(x+1)}dx = & 0\,.
            \end{cases}
            \end{align*}
For $t\approx 2.60$, $a_1$, $b_1$ and $a_2$ collide at the point $b \approx 0.07$, creating a double root of $B$, which corresponds with the second phase transition (type II singularity: merger of cuts or closing of a gap).

Finally, as $t$ continues growing up, scenario $(a2)$ takes place. By \eqref{a2}, $a_1$ (which, according to our notation, is now the name of the former $a_3$) continues increasing, while the double root of $B$ splits into two complex conjugated roots, with a decreasing real part. This couple of imaginary roots are going away from the real axis since they are inside the circle $D_t$ \eqref{circleDt}. The equations to determine the parameters of the equilibrium measure are
            \begin{align*}
            \begin{cases}
             a_1+2b_1+2b_2+2 = & -2\beta\,,\\
             a_1^2-8(1+b_1)(1+b_2)-4a_1(b_1+b-2+1) = & -8(\gamma-t)\,,\\
             \sqrt{1+a_1}(1+b_1)(1+b_2) = & \gamma \,,
             \end{cases}
            \end{align*}
and $b_1$ and $b_2$ continues going away from the real axis while they are inside the circle $D_t$, but at some moment they get out the circle and begin to approach the real axis and at the left side of $-1$. The roots $b_1$ and $b_2$ collide at $x\approx -1.89$ and $t\approx 43.94$, become real and begin to separate (scenario (a1) takes place); finally, $b_1$ converges to $-1$ and $b_2$ goes to $-\infty$ according to \eqref{a1} and Theorem \ref{thm:end}.

\end{remark}

\subsection{Dynamics with respect to the variation of the prescribed charge}
The aim of this section is to study the evolution of the equilibrium measure and its support when other parameter of the problem varies, in particular, the position of the prescribed charge, especially dealing with the case where $v\searrow 0$ (connecting with the so-called generalized Gaussian-Penner models, see \cite{Deo1993} and \cite{Deo2002}, to only cite a few).

In order to do it, it seems convenient to consider the variation of the equilibrium measure with respect to parameter $v$, in place of $t$ as in the previous section. For it, Theorem \ref{parameters} above will be used.

Since in this case we mainly deal with the situation when $v\searrow 0$, it is not suitable making use of the external field in the simplified form \eqref{simplified}. Now, it is better to consider the external field \eqref{semiaxis}, setting $\beta = 2b$ and $\gamma = 2c$, that is,
\begin{equation}\label{penner}
\phi (x) = \frac{x^2}{2} + \beta x + \gamma \log (x+v)\,, \;\; x\in (0,+\infty)\,,
\end{equation}
with $\beta,\gamma \in \R$ and $v>0$. We are interested in studying the evolution of the equilibrium measure $\nu$ in the external field \eqref{penner} and, in particular, our aim is describing the dynamics when $v\searrow 0\,.$ For the sake of simplicity, we also use the notation $\dot{f}$ for denoting the derivative of a function $f$ with respect to the parameter $v$: $\dsty \dot{f} = \frac{\partial f}{\partial v}\,.$

Now, our main result in this subsection is announced.

\begin{theorem}\label{thm:mainv}
Let $\nu_v = \nu_{t,v}$ be the equilibrium measure in $[0,+\infty)$ in the external field $\phi$ given by \eqref{penner}.

\begin{itemize}

\item [(a)] If $\gamma < 0$, then $$\lim_{v\searrow 0}\,\nu_v = \sigma\,,$$ where $$\sigma'(x) = \frac{1}{\pi}\,\frac{(x-b_1)\,\sqrt{(x-a_2)(a_3-x)}}{x}\,,\;x\in (a_2,a_3)\,,$$
with $b_1,a_2,a_3$ given by \eqref{admiss}-\eqref{b1}, and the convergence holds in the weak-* sense. Indeed, $\nu_v$ can be continued analytically for $v<0$.

\item [(b)] If $\gamma > 0$, we have that
\begin{equation}\label{limnoadmiss}
\lim_{v\searrow 0}\,\nu_v = \begin{cases} t\,\delta_0\,,\; & \text{if}\;t\leq \gamma\,,\\ \gamma\,\delta_0 + (t-\gamma)\,\rho\,,\; & \text{if}\;t > \gamma\,,
\end{cases}
\end{equation}
where $\rho$ is the unit equilibrium measure of $[0,+\infty)$ in the presence of $\psi (x) = \frac{x^2}{2}\,+\beta x\,.$

\end{itemize}

\end{theorem}

The rest of this Subsection is devoted to the proof of Theorem \ref{thm:mainv}.

Let us assume we are given fixed $\beta,\gamma \in \R$ and $t>0$, and suppose that for $v$ in a certain open interval in $(0,+\infty)$, $S_v = S_{t,v}$ consists of a single interval $[0,a_1]$. Thus, from Theorem \ref{parameters} and taking into account the characterization \eqref{condparam} and that $\dsty \dot{\phi} = \frac{\gamma}{x+v}$, one has:
\begin{equation}\label{derivparam1}
\frac{\partial}{\partial t}\,\left(\frac{(z-b_1)(z-b_2)}{z+v}\,\sqrt{\frac{z-a_1}{z}}\right)\,=\,\frac{H(z)}{(z+v)^2\,\sqrt{z(z-a_1)}}\,,
\end{equation}
where $H(z) = C(z-h_1)\,,$ for a certain constant $C$ and $h_1\in (0,a_1)$. In a similar way, if for this interval of values of $v$, $S_{v}$ consists of a single interval $[a_2,a_3]$ with $0<a_2<a_3\,,$  then:
\begin{equation}\label{derivparam2}
\frac{\partial}{\partial t}\,\left(\frac{z-b_1}{z+v}\,\sqrt{(z-a_2)(z-a_3)}\right)\,=\,\frac{H(z)}{(z+v)^2\,\sqrt{(z-a_2)(z-a_3)}}\,,
\end{equation}
where $H(z) = C(z-h_1)\,,$ for a certain constant $C$ and $h_1\in (a_2,a_3)$. Finally, when the support is comprised of the union of two disjoint intervals $[0,a_1]\cup [a_2,a_3]\,,$
with $0<a_1<a_2<a_3\,,$  then:
\begin{equation}\label{derivparam3}
\frac{\partial}{\partial t}\,\left(\frac{z-b_1}{z+v}\,\sqrt{\frac{(z-a_1)(z-a_2)(z-a_3)}{z}}\right)\,=\,\frac{H(z)}{(z+v)^2\,\sqrt{z(z-a_1)(z-a_2)(z-a_3)}}\,,
\end{equation}
where $H(z) = C(z-h_1)(z-h_2)\,,$ for a certain constant $C$ and two positive real roots, $h_1<h_2$, one of them in $(a_1,a_2)$, and the other in $(0,a_3)$.

First, consider the case where $\gamma < 0$, that is, when the Gaussian-type model is perturbed by a ``repulsive'' charge at the point $x=-v\in (-\infty,0)$. In this case, it is easy to see that the field corresponding to the limit case $v=0$, that is, $$\phi_0(x) = \begin{cases} \frac{x^2}{2} + \beta x +\gamma \log x\,,& x\in (0,+\infty)\,, \\  +\infty\,, & x\in (-\infty,0]\,, \end{cases}$$ is an admissible field in $\R$. Thus, for any pair $\beta \in \R, \gamma <0$ and $t>0$ there exists the corresponding equilibrium measure $\nu_0$, compactly supported in $(0,+\infty)$. $S_0$, the support of $\nu_0$ may be determined taking into account \eqref{Cauchytr}, where now the rational function will be of the form (like in scenario (b) above): $$R_t(z) = \frac{z-b_1}{z}\,\sqrt{(z-a_2)(z-a_3)}\,,\,b_1<0<a_2<a_3\,.$$ Indeed, equating coefficients in the above expression and taking residues at the origin, the following system of nonlinear equations holds:
\begin{equation*}\label{nonlinearsyst}
\begin{cases} 2b_1+a_2+a_3 +2\beta & = 0\,, \\ b_1^2+2b_1\,(a_2+a_3) + a_2 a_3 + 2(t-\gamma) - \beta^2& = 0\,, \\ b_1^2\,a_2 a_3 & = \gamma^2\,, \end{cases}
\end{equation*}
from which we obtain that
\begin{equation}\label{admiss}
\begin{split}
a_2  & = -(b_1+\beta) -\,\sqrt{(b_1-\beta)^2-\beta^2-3b_1^2+2(t-\gamma)}\,>0\,, \\
a_3  & = -(b_1+\beta) +\,\sqrt{(b_1-\beta)^2-\beta^2-3b_1^2+2(t-\gamma)}\,>a_2>0\,,
\end{split}
\end{equation}
where $b_1$ must be a negative real root of the equation
\begin{equation}\label{b1}
p(y) = 3y^4 + 4\beta y^3 + (\beta^2+2(\gamma -t)) y^2 - \gamma^2 = 0\,,
\end{equation}
such that the corresponding values for $a_2,a_3$ in \eqref{admiss} are positive.
By the Descartes rule, it is easy to see that for $\beta \leq 2\sqrt{t-\gamma}$, polynomial $p$ in \eqref{b1} has a unique negative real root. But for $\beta > 2\sqrt{t-\gamma}$, \eqref{b1} could have, in principle, three negative real roots. However, in this case, \eqref{admiss} shows that the fact that $a_2,a_3>0$ implies that $b_1<-\beta$ and, thus, we have to look only for negative roots such that $y<-\beta <-2\,\sqrt{t-\gamma}$. But taking derivatives in \eqref{b1} and applying Rolle's Theorem, we obtain that the lower root of the polynomial $p'$ is $y_0 = -\frac{\beta}{2}\,-\sqrt{\frac{\beta^2}{12}+\frac{t-\gamma}{3}}\,>-\beta\,.$ So, in this case, if $p$ has $3$ negative real roots, then only one of them is admissible ($<-\beta$). Therefore, the recipe to find the value of $b_1$ is as follows. If $\beta \leq \sqrt{2(t-\gamma)}$, $b_1$ is the unique negative real root of $p$ in \eqref{admiss}; on the other hand, if $\beta > \sqrt{2(t-\gamma)}$, $b_1$ is the unique root of $p$ belonging to the interval $(-\infty,-\beta)\,.$ Finally, it is easy to check from \eqref{admiss} that $0<a_2<a_3\,.$

For instance, in the typical case where $\gamma = -1\,,\,\beta = 0$ and $t = 1\,,$ we have that $b_1$ is the unique negative root of the equation
$3y^4-4y^2-1 = 0\,,$
which is given by
$b_1 = -\sqrt{\frac{2+\sqrt{7}}{3}}\,=\,-1.2444...$ Thus, $a_2 = 0.294...$, $a_3 = 2.1946...$

Now, take into account that expression \eqref{derivparam2} provides information about the evolution of the support when $v\searrow 0$. In this sense, it is easy to obtain:
\begin{equation}\label{constants1}
C = -\gamma \,\frac{2v+a_2+a_3}{2\sqrt{(v+a_2)(v+a_3)}}\,>\,0\,,\;h_1 = \frac{v(a_2+a_3)+a_2a_3}{2v+a_2+a_3}\in (a_2,a_3)\,.
\end{equation}
On the other hand, taking residues at $z=-v$, we have that $\dsty \gamma = (v+b_1)\,\sqrt{(v+a_2)(v+a_3)}$ and, hence, it yields $v+b_1<0\,.$ Now, from this fact and \eqref{constants1} we get
\begin{equation*}\label{dynamics2}
\begin{split}
\dot{a_2} & = -\,\frac{2H(a_2)}{(a_2-a_3)(a_2-b_1)(a_2+v)}\,<0\,,\;\dot{a_3} = -\,\frac{2H(a_3)}{(a_3-a_2)(a_3-b_1)(a_3+v)}\,<0\,,\\
\dot{b_1} & = -\,\frac{2H(b_1)}{(b_1-a_2)(b_1-a_3)(b_1+v)}\,>0\,,
\end{split}
\end{equation*}
and, thus, we have that $a_2,a_3$ are increasing, while $b_1$ is decreasing, as $v\searrow 0$. Since the limit external field for $v=0$ is an admissible field, this renders the proof of part (a). Note that the equilibrium problem can be continued analytically for $v<0$.

Now, consider the situation when $\gamma > 0$. In this case, the corresponding limit external field $\phi_0$ (taking $v=0$ in \eqref{penner}) is not an admissible field, as in the previous case. On the other hand, since $\lim_{v\searrow 0}\,\phi(x) = -\infty$, it is clear that for $v$ small enough, $\dsty \min_{x\in [0,+\infty)}\,\phi (x) = \phi (0)\,.$ Thus, we can assure that for a small enough value of $v$, the support of the equilibrium measure is given either by a single interval $[0,a_1]$ or by the union of two disjoint intervals $[0,a_1]\cup [a_2,a_3]$.

Let us show, now, the variation of the endpoints (and the other zeros of the density of the equilibrium measure) when parameter $v\searrow 0$. From the discussion above, it
is clear that $0$ belongs to the support for such values of $v$.

Suppose, first, that we have $S_{v} = [0,a_1]$. Then, making use of \eqref{derivparam1} and taking into account that $\dsty 0<\gamma = (v+b_1)(v+b_2)\,\sqrt{\frac{v+a_1}{v}}\,$ and, hence, $(v+b_1)(v+b_2)>0$, it yields:
$$C = -\gamma \,\frac{2v+a_1}{2\sqrt{v(v+a_1)}}\,<\,0\,,\;h_1 = \frac{va_1}{2v+a_1}\in (0,a_1)\,,$$
what implies that
\begin{equation*}\label{1Ca1}
\dot{a_1} = -\,\frac{2H(a_1)}{(a_1-b_1)(a_1-b_2)(a_1+v)}\,>0\,,
\end{equation*}
that is, the endpoint $a_1$ always decreases as $v\searrow 0$. In a similar fashion we obtain for the cases where $b_1,b_2 \in \R$ (recall that we always denote by $b_1$ the closest zero of $B$ to the support $[0,a_1]$):
\begin{equation*}\label{1Cb1b2}
\begin{cases} \dot{b_1}<0\;\;\text{and}\;\;\dot{b_2}>0\,,\;\; & \text{if}\;\;a_1<b_1<b_2 \;\;\text{or}\;\;b_2<b_1<-v<0<a_1\,, \\
\dot{b_1}>0\;\;\text{and}\;\;\dot{b_2}<0\,,\;\; & \text{for}\;\;-v<b_1<b_2<0<a_1\,. \end{cases}
\end{equation*}

On the other hand, if scenario (c) takes place and we make use of \eqref{derivparam3}, it yields for the derivatives of the endpoints with respect to the parameter $v\,$:
\begin{equation}\label{dynamicsb}
\begin{split}
\dot{a_1} & = -\,\frac{2H(a_1)}{(a_1-b_1)(a_1-a_2)(a_1-a_3)(a_1+v)}\,,\;\dot{a_2} = -\,\frac{2H(a_2)}{(a_2-a_1)(a_2-b_1)(a_2-a_3)(a_2+v)}\,,\\
\dot{a_3} & = -\,\frac{2H(a_3)}{(a_3-a_1)(a_3-b_1)(a_3-a_2)(a_3+v)}\,,\dot{b_1} = -\,\frac{H(b_1)}{(b_1-a_1)(b_1-a_2)(b_1-a_3)(b_1+v)}\,.
\end{split}
\end{equation}
Let us see the sign of the leading coefficient of $H$. Since
$$\frac{H(z)}{\sqrt{zA(z)}(z+v)^2}=\frac{-\gamma}{(z+v)^2}+O(1)\,,\qquad \text{as }z\to-v\,,$$
we have that
$$H(-v)=-\sqrt{-vA(-v)}\gamma<0\,,$$
and taking into account that the two roots of $H$ are in $(0,a_3)$, we get that $C>0$. In particular, this implies that $\dot{a}_3>0$, that is, $a_3$ decreases as $v$ decreases. On the other hand, the behavior of $a_1$ and $a_2$ is not so clear because we do not know the exact location of one of the roots of $H$.

Now, we are concerned with the limit equilibrium measure as $v\searrow 0$. Thus, suppose we are given fixed $\gamma, t >0$ and $\beta \in \R$ and let us study the evolution of the equilibrium measure in the external field \eqref{penner} when $v\searrow 0$.

In order to do it, we first show in the following lemma that no phase transition occurs when $v$ is sufficiently small. In this sense, take into account that we do not know a priori if an infinite amount of phase transitions may occur as $v\searrow 0$. The following lemma ensures this cannot happen and thus, we conclude that for $v$ small enough, the support always consists of an unique interval $[0,a_1]$ or of two intervals $[0,a_1]\cup[a_2,a_3]$.

Then, we have:
\begin{lemma}\label{lem:regular}
There exists $v_0 = v_0 (\beta, \gamma, t) <0$ such that
\begin{itemize}
    \item[i)] $h_1< a_1$ for any $v<v_0$ and
    \item[ii)] there is no phase transition in the interval $(0,v_0)$.
\end{itemize}
\end{lemma}
\begin{proof}
\begin{itemize}
    \item[i)] If the support consist of a single interval, it is obvious. Thus, let us suppose that $S_{v}=[0,a_1]\cup[a_2,a_3]$. From Theorems 1.1 and 1.3, taking Laurent series expansion at $z=-v\,,$ it holds
        $$-\widehat{\omega}(z)+\dot{\phi'}(z)=\frac{P(z)}{\sqrt{zA(z)}(z+v)^2}=
        \frac{-\gamma}{(z+v)^2}+O(1)\,,$$
        where $\omega = \frac{\partial \nu_v}{\partial v}$.
        Thus, imposing that the derivative with respect to $z$ at $z=-v$ must vanish, we get from \eqref{derivparam3} that
        $$\frac{2H'(-v)}{H(-v)}=\frac{A(-v)-vA'(-v)}{-vA(-v)}\,,$$
        and, then,
        $$\frac{2}{v+h_1}+\frac{2}{v+h_2}=\frac{1}{v}+\frac{1}{v+a_1}+\frac{1}{v+a_2}+\frac{1}{v+a_3}\,.$$
        Since we can bound
        $$\frac{4}{v+h_1}>\frac{2}{v+h_1}+\frac{2}{v+h_2}=\frac{1}{v}+\frac{1}{v+a_1}+\frac{1}{v+a_2}+\frac{1}{v+a_3}>\frac{1}{v}\,,$$
        we see that $h_1<3v$.

        Suppose now that the assertion {\emph i)} is false. Then there it would exist $v$ as small as we want such that $a_1<h_1<3v$. But the Laurent series expansion
        \begin{equation}\label{Laurent}
        -\widehat{\nu_v}(z)+\phi'(z)=\frac{\sqrt{A(z)}B(z)}{\sqrt{z}(z+v)}=\frac{\gamma}{z+v}+O(1)
        \end{equation}
        yields $A(-v)B^2(-v)=-v\gamma^2$ and, then,
        $$\left|(-v-a_2)(-v-a_3)(-v-b_1)^2\right|=\frac{v\gamma^2}{v+a_1}>\frac{\gamma^2}{4}.$$
        In particular, for such $v$, $b_1$ cannot approach the origin. On the other hand, from \eqref{eqcond}, $b_1$ can be determined as
        \begin{equation}\label{center}
        b_1=\frac{\int_{a_1}^{a_2}\frac{x}{x+v}\sqrt{\frac{(x-a_1)(x-a_2)(x-a_3)}{x}}dx}
          {\int_{a_1}^{a_2}\frac{1}{x+v}\sqrt{\frac{(x-a_1)(x-a_2)(x-a_3)}{x}}dx}\,,
        \end{equation}
        in such a way that if $v,a_1\searrow 0$, the denominator of \eqref{center} tends to $+\infty$ while the numerator remains bounded, that is, $b_1$ would tend to $0$; but it is not possible as we have seen previously. The contradiction comes from the assumption that we can find $v$ as small as we want with $a_1<h_1$. Hence, {\emph i)} is proved.

    \item[ii)] Using {\emph i)}, it is clear that for $v$ small enough, $\omega$ is negative in $[0,h_1]\subset[0,a_1]$ and positive in $[h_1,a_1]$ or $[h_1,a_1]\cup[a_2,a_3]$, depending on the number of cuts of $S_{v}$; that is, $\nu$ is increasing for $x<h_1$ and decreasing for $x>h_1$. Then, in principle, the unique phase transitions allowed are the ``death'' of the cut $[a_2,a_3]$ (in case of there were two cuts) and the split of $[0,a_1]$ into two intervals (if there was a single cut). So, the unique possibility to have infinite phase transitions in a finite range of $v$ would be the alternation of these two settings. But, if such a situation takes place, between the disappearance of the cut $[a_2,a_3]$ and the posterior split of $[0,a_1]$ the support would consist of a single interval and the roots of $B$ would have to move to the left to produce the next split. But this is not possible since the arithmetic mean of the roots of $B$ is increasing during this phase. Indeed, the root $a_1$ is decreasing and comparing the $O(1)\,$ - term of $(-\widehat{\nu_v}+\phi')$ at $z=\infty$ in \eqref{Laurent}, we have
        $$a_1+2b_1+2b_2+2=2\beta\,,$$
        and, hence, $\dot{b}_1+\dot{b}_2<0$, that is $b_1+b_2$ is increasing.

        Therefore, for $v$ small enough, if a cut disappears when $v = V>0$, then $S_{v}$ will consist of a single interval for any $v\leq V$.

\end{itemize}
\end{proof}

As a consequence, if $S_{v}$ consists of one interval for $v$ small enough, then this setting remains for any $v\searrow 0$ and $a_1$ decreases; and if $S_{v}$ consists of two intervals, this situation also holds for $v\searrow 0$ and
$$\dot{a}_1>0\,,\qquad \dot{a}_2<0\,,\qquad \dot{a_3}>0\,,$$
that is, $a_1$ approaches the origin while $a_2$ and $a_3$ approach each other.

Now, we can complete the proof of part (b) in Theorem \ref{thm:mainv}.

Recall that for $\nu$ small enough, there are only two options for the support $S_{v}$, namely, $S_{v} = [0,a_1]\,,\;0<a_1\,,\;$ or  $S_{v} = [0,a_1]\cup [a_2,a_3]\,,\;0<a_1<a_2<a_3.$

In the first case, we are in scenario (a) above and, thus, from \eqref{scen(a)} we have:
\begin{equation}\label{nonlsyst(a)}
\begin{cases} a_1+2(b_1+b_2) & = -2(\beta + v)\,,\\(b_1+b_2)^2+2a_1(b_1+b_2)+2b_1b_2 & = (\beta + v)^2 + 2(\beta v+\gamma-t)\,,\\(v+b_1)^2(v+b_2)^2(v+a_1) & = \gamma^2 v\,.\\
\end{cases}
\end{equation}
Third equation in \eqref{nonlsyst(a)} shows that when $v\searrow 0$, then $a_1\searrow 0$ or $b_1b_2 \rightarrow 0$ and, hence, $b_1 \rightarrow 0$.

If $a_1\searrow 0$, taking limits in the first two equations in \eqref{nonlsyst(a)}, we have for the limits $b_1^0,b_2^0$ of $b_1,b_2$, respectively:
\begin{equation*}\label{limita1}
b_1^0+b_2^0 = -\beta\,,\;(b_1^0+b_2^0)^2+2b_1^0b_2^0 = \beta^2 + 2(\gamma-t)\,,
\end{equation*}
and so,
\begin{equation}\label{limita1}
b_1^0+b_2^0 = -\beta\,,\;b_1^0b_2^0 = \gamma-t\,.
\end{equation}
On the other hand, if $a_1 \nrightarrow 0$ and $b_1\rightarrow 0$, we have that the respective limits $a_1^0,b_2^0$ satisfy the nonlinear system:
\begin{equation}\label{limita2}
a_1^0 + 2b_2^0 = -2\beta\,,\;b_2^0 (b_2^0+2a_1^0) = \beta^2 + 2(\gamma-t)\,,
\end{equation}
where we are looking for solutions $a_1^0 > 0$ and $b_2^0 \leq 0\,.$

Now, let us consider the second option, that is, when $S_{v} = [0,a_1]\cup [a_2,a_3]$. This situation is drawn in scenario (c) and, hence, \eqref{scen(c)} yields:
\begin{equation}\label{nonlsyst(c)}
\begin{cases} a_1+a_2+a_3+2b_1 & = -2(\beta +v)\,,\\b_1^2+2b_1(a_1+a_2+a_3)+a_1a_2+a_1a_3+a_2a_3 & = (\beta +v)^2 + 2(\beta v+\gamma-t)\,,\\(v+b_1)^2(v+a_1)(v+a_2)(v+a_3) & = \gamma^2 v\,,\\
\end{cases}
\end{equation}
with $0<a_1<b_1<a_2<a_3$ and $b_1$ given by \eqref{center}. Thus, by the same argument used in the proof of Lemma \ref{lem:regular}, we conclude that necessarily $b_1\searrow 0$ too.

Now, denoting by $a_2^0,a_3^0$ the limits of $a_2,a_3$, respectively, the following system holds:
\begin{equation}\label{systc}
a_2^0+a_3^0 = -2\beta \,,\;a_2^0\,a_3^0 = \beta ^2 + 2(\gamma -t)\,,
\end{equation}
and we have to look for possible solutions of \eqref{systc} such that $0\leq a_2^0 \leq a_3^0$.

Once all the possible settings have been reviewed, let us suppose now that $t\leq \gamma\,.$ It is easy to check that in this case only scenario (a) with $a_1\searrow 0$ is feasible. Indeed, system \eqref{limita1} yields:
\begin{equation*}
b_1^0\,,\,b_2^0 \begin{cases} \geq 0\,,\;\; & \text{if}\;\;\beta \leq -2\,\sqrt{\gamma -t}\,, \\
< 0\,,\;\; & \text{if}\;\;\beta > 2\,\sqrt{\gamma -t}\,, \\ \in \C \setminus \R\,,\,b_2^0 = \overline{b_1^0}\,,\;\; & \text{if}\;\;|\beta | < 2\,\sqrt{\gamma -t}\,.
\end{cases}
\end{equation*}
In all these instances, the limit measure is $\sigma = t\delta_0\,,$ since
$$\lim_{v\searrow 0}(-\widehat{\nu}(z)+\phi'(z)) = \frac{(z-b_1^0)(z-b_2^0)}{z}\,,$$ and, thus, setting $\dsty \lim_{v\searrow 0} \nu_v = \sigma\,,$ it yields,
$$\widehat{\sigma}(z) = z + \beta + \frac{\gamma}{z} - \frac{(z-b_1^0)(z-b_2^0)}{z}\,,$$ and then, by \eqref{limita1},
$$\widehat{\sigma}(z) = z + \beta + \frac{\gamma}{z} - \frac{z^2+\beta z+\gamma -t}{z}\,=\,\frac{t}{z}\,.$$ This proves the first assertion in \eqref{limnoadmiss}.

Suppose, now, that $t>\gamma \,.$ In this case, scenarios (a), with $a_1\searrow 0$ and $b_1\rightarrow 0$, and (c) are both feasible. Indeed, studying the system \eqref{limita2}, corresponding to the scenario (a) when $a_1\searrow 0$ and $b_1\rightarrow 0$, we find solutions
\begin{equation}\label{ab}
a_1^0 = -\frac{2\beta}{3} + \frac{2}{3}\,\sqrt{\beta^2+6(t-\gamma )}\,>0\,,\;
b_2^0 = -\frac{2\beta}{3} - \frac{1}{3}\,\sqrt{\beta^2+6(t-\gamma )}\,<0\,,
\end{equation}
provided that $t-\gamma \geq 0$ and $\dsty \beta < -\sqrt{2(t-\gamma)}\,.$ In this case, we have that $$\widehat{\sigma}(z) = \,\frac{t}{z}\,+z + \beta\,-(x-b_2^0)\,\sqrt{\frac{z-a_1^0}{z}}\,,$$ and from \eqref{ab} it is not hard to check that $$\widehat{\rho}(z) = z + \beta\,-(x-b_2^0)\,\sqrt{\frac{z-a_1^0}{z}}$$ is the Cauchy transform of the equilibrium measure of total mass $t-\gamma >0$ in $[0,+\infty)$ in the external field $\psi(z) = \frac{z^2}{2} + \beta z\,.$ Observe that in this case the limit measure has a Dirac mass at the origin plus an absolutely continuous part whose support also has the origin as an endpoint. This establishes the second assertion in \eqref{limnoadmiss} for the case where $\dsty \beta < -\sqrt{2(t-\gamma)}\,.$

Thus, it only remains to analyze what happens when $t>\gamma \,$ and $\dsty \beta < -\sqrt{2(t-\gamma)}\,.$
But, for this case, system \eqref{systc} provides solutions
$$0 \leq a_2^0 = -\beta - \sqrt{2(t-\gamma)}\,\leq a_3^0 = -\beta + \sqrt{2(t-\gamma)}\,.$$

In this case, we have that $$-\widehat{\sigma}(z)\,+\,\phi'(z) = \sqrt{(z-a_2^0)(z-a_3^0)} = \sqrt{(z+\beta)^2-2(t-\gamma)}\,,$$
and, thus, $$\widehat{\sigma}(z) = \,\frac{\gamma}{z}\,+\,\left(z+\beta\,-\,\sqrt{(z+\beta)^2-2(t-\gamma)}\right) = \,\frac{\gamma}{z}\,+\,\frac{2}{z+\beta\,+\,\sqrt{(z+\beta)^2-2(t-\gamma)}}\,.$$

Observe that in this last situation, for $t>\gamma$ the absolutely continuous part of the measure has not the origin as an endpoint.

This renders the proof of the second part of Theorem \ref{thm:mainv}.

\vspace{1cm}

R. Orive,

Universidad de La Laguna, Canary Islands, Spain.

rorive@ull.es.

Research partially supported by Ministerio de Ciencia e Innovaci\'{o}n under grant MTM2011--28781.

\vspace{.5cm}

J. S\'{a}nchez Lara,

Universidad de Granada, Spain.

jslara@ugr.es.

Partially supported by research project of Junta de Andaluc\'{\i}a under grant FQM229.
\end{document}